\documentclass[12pt]{amsart}
\numberwithin{equation}{section}
\usepackage{amsmath,amsthm,amsfonts,amscd,eucal,mathabx}

\usepackage{xypic}
\usepackage{graphicx}
\usepackage{amssymb}
\def\ca{{\mathcal A}}
\def\cb{{\mathcal B}}

\def\cd{{\mathcal D}}

\def\cf{{\mathcal F}}

\def\ch{{\mathcal H}}

\def\ck{{\mathcal K}}

\def\cs{{\mathcal S}}

\def\ga{{\mathfrak A}}
\def\gb{{\mathfrak B}}

\def\ba{{\mathbb A}}

\def\bc{{\mathbb C}}

\def\bm{{\mathbb M}}
\def\bn{{\mathbb N}}

\def\br{{\mathbb R}}
\def\bt{{\mathbb T}}

\def\bz{{\mathbb Z}}

\def\a{\alpha}
\def\b{\beta}
  \def\G{\Gamma}
\def\d{\delta}  \def\D{\Delta}
\def\eeps{\epsilon}
\def\eps{\varepsilon}

\def\l{\lambda} 
\def\k{\kappa}
\def\m{\mu}

\def\n{\nu}
\def\r{\rho}
\def\s{\sigma} \def\S{\Sigma}
\def\t{\tau}
\def\f{\varphi}  \def\F{\Phi}
\def\th{\theta} 
\def\om{\omega}

\def\id{\hbox{id}}
\def\ker{\hbox{Ker}}

\newtheorem{thm}{Theorem}[section]
\newtheorem{lem}[thm]{Lemma}

\newtheorem{prop}[thm]{Proposition}

\theoremstyle{definition}
\newtheorem{rem}[thm]{Remark}
\newtheorem{defin}[thm]{Definition}

\def\sign{\mathop{\rm sign}}

\newcommand{\ty}[1]{\mathop{\rm {#1}}}
\def\di{{\rm d}}

\def\idd{{1}\!\!{\rm I}}

\DeclareMathAlphabet{\mathpzc}{OT1}{pzc}{m}{it}

\begin{document}

\title[noncommutative torus]
{fourier analysis for type III representations of the noncommutative torus}
\author{Francesco Fidaleo}
\address{Francesco Fidaleo\\
Dipartimento di Matematica \\
Universit\`{a} di Roma Tor Vergata\\
Via della Ricerca Scientifica 1, Roma 00133, Italy} \email{{\tt
fidaleo@mat.uniroma2.it}}

\date{\today}


\begin{abstract}
For the noncommutative 2-torus, we define and study Fourier transforms arising from representations of states with central supports in the bidual, exhibiting a possibly nontrivial modular structure ({\it i.e.} type $\ty{III}$ representations).

We then prove the associated noncommutative analogous of Riemann-Lebesgue Lemma and Hausdorff-Young Theorem. In addition, the $L^p$- convergence result of the Cesaro means ({\it i.e.} the Fejer theorem), and the Abel means reproducing the Poisson kernel are also established,
providing inversion formulae for the Fourier transforms in $L^p$ spaces, $p\in[1,2]$.

Finally, in $L^2(M)$ we show how such Fourier transforms "diagonalise" appropriately some particular cases of modular Dirac operators, the latter being part of a one-parameter family of modular spectral triples naturally associated to the previously mentioned non type $\ty{II_1}$ representations.
\vskip0.1cm\noindent
\end{abstract}

\maketitle

\centerline{Dedicato a Maddalena Briamonte}

\section{introduction}
\label{sec1}

\noindent
The present paper is devoted to extend the Fourier analysis to non-type $\ty{II_1}$ representations of the noncommutative 2-torus, that is when the underlying "measure" on the noncommutative manifold under consideration is not the canonical trace. This would include also the type $\ty{II_1}$ case when the underlying measure relative to the trace is deformed by an inner, possibly unbounded, density. 

The argument of the present paper falls into the standard lines of investigation of operator algebras and noncommutative harmonic analysis, perhaps having many potential interactions with other branches of mathematics, and quantum physics.  Due to the crucial role played by the necessarily nontrivial modular data, such a topic can be viewed as a generalisation of the analysis developed in \cite{CXY} for the representation of the noncommutative torus associated to the tracial state.

For the noncommutative 2-torus $\ba_\a$, we have recently shown in \cite{FS} that it is possible to construct explicitly type $\ty{III}$ representations, at least when $\a$ is Liouville number. In addition, if $\a$ is one such number with a faster approximation property by rationals, it is also possible to construct modular spectral triples that, following \cite{CM}, are obtained by deforming the untwisted Dirac operator and the associated commutator by using the Tomita modular operator which is neither bounded nor inner.

As a quick introduction, we recall that the study of Connes' noncommutative geometry grew impetuously in the last decades in view of several potential applications to mathematics and physics.
The main ingredient is the so-called spectral triple, which is the candidate to encode the most important properties of the underlying ``noncommutative manifold''. For an exhaustive explanation of the fundamental role played by the spectral triples in noncommutative geometry, the reader is referred to the seminal monograph \cite{C}, the expository paper \cite{CPR}, and the literature cited therein.

Having in mind such possible applications, the seminal paper \cite{CM} also pursued the plane to exhibit twisted spectral triples for which the Dirac operator and the associated derivation are deformed by directly using the Tomita modular operator, provided the last is nontrivial. In that definition of twisted spectral triple, also the Tomita conjugation plays a role, which can be identified as a kind of "charge conjugation" well known in quantum physics. The twisted spectral triples considered in the above mentioned paper were then called modular. It is also clear that such modular spectral triples can play a role in constructing new noncommutative geometries in a type $\ty{III}$ setting, by emphasising
the need to take the modular data into account. 

One of the most studied examples in noncommutative geometry is indeed the noncommutative 2-torus $\ba_\a$ (see {\it e.g.} \cite{B}), since it is a quite simple model, though highly nontrivial. It is related to the discrete Canonical Commutation Relations, and can be considered as a quantum deformation of the classical $2$-torus 
$\bt^2$, according to the angle $2\pi\a$ entering in the definition of the symplectic form involved in the construction. When $\a$ is irrational, the $C^*$-algebra $\ba_\a$ is simple and admits a unique, necessarily faithful, trace $\t$. Such a trace represents the natural noncommutative generalisation of the Lebesgue measure on the undeformed 2-torus $\bt^2$. 
Therefore, the associated GNS representation $\pi_\t$ gives rise to a von Neumann factor $\pi_\t(\ba_\a)''$, isomorphic to the Murray-von Neumann hyperfinite type $\ty{II_1}$ factor, which can be viewed as the noncommutative counterpart of the von Neumann group algebra $L^\infty(\bt^2,m\times m)$, $m$ being the Haar measure ({\it i.e.} the normalised Lebesgue measure) on the unit circle. Hence, $\ba_\a$ is not a type $\ty{I}$ $C^*$-algebra for irrational $\a$, and therefore it must exhibit also type $\ty{II_\infty}$ and $\ty{III}$ representations.
The preliminary step to construct explicit non type $\ty{II_1}$ representations of the noncommutative torus, and consequently examples of modular spectral triples, was carried out in \cite{FS}.

In order to develop the Fourier analysis for these new models ({\it i.e.} when the underlying "measure" is not the canonical trace) associated to the noncommutative torus, we introduce and investigate two essentially different Fourier transforms naturally associated to the left and right embeddings, which are particular cases of the ones
$$
L^\infty\big(\pi(\ba_\a)''\big)\equiv\pi(\ba_\a)''\hookrightarrow\iota^\th_{\infty,1}\big(\pi(\ba_\a)''\big)\subset \pi(\ba_\a)''_*\equiv L^1\big(\pi(\ba_\a)''\big)
$$ 
given in \eqref{em10}, due to the non-triviality of the modular structure. Here, $\th=0,1$ correspond indeed to the left and right embeddings, respectively. Notice that our investigation appears new also 
for any representation of type $\ty{II_1}$ (including all examples treated in \cite{CM}) and $\ty{II_\infty}$, when the modular group is non trivial, even if "inner". 

Our construction can also be generalised to the noncommutative deformation $\ba_{\boldsymbol\a}$ of the $d$-dimensional torus $\bt^d$, ${\boldsymbol\a}$ being a real skew-symmetric $d\times d$-matrix ({\it e.g.} \cite{CXY}, Section 2.2), after constructing non type $\ty{II_1}$ representations for such models, following the method developed in \cite{FS}.

For such noncommutative examples of Fourier transforms, we prove the analogous of Riemann-Lebesgue Lemma and Hausdorff-Young Theorem. 
In addition, for $p\in[1,2]$ we
establish an inversion formula arising from the Cesaro mean, that is  a noncommutative version of the Fejer Theorem. Similar results can be obtained for the Abel average which is associated with the Poisson kernel, and also for the other cases mentioned in \cite{CXY}.

The last part of the present paper is devoted to a useful application of the Fourier analysis developed here, to new modular spectral triples for the noncommutative 2-torus. Firstly, for non type $\ty{II_1}$ cases, we extend the construction of deformed Dirac operators and corresponding modular spectral triples in \cite{FS} to a one-parameter family for $\eta\in[0,1]$, where the case $\eta=0$ is nothing but the one studied in the previous mentioned paper. Secondly, in $L^2(M)$ we show how such Fourier transforms "diagonalise" appropriately the particular cases of modular Dirac operators, corresponding to $\eta=0,1/2,1$, part of such a one-parameter family of modular spectral triples.

To conclude the introduction, we recall other interesting questions, such as those listed in \cite{CXY}, which we leave open in the present analysis, postponing them for a future possible investigation. Among them, we mention the construction and the study of the noncommutative Hardy spaces relative to the situations emerging from the present paper. Other important problems are those connected to the  
spaces of the so-called {\it Fourier multipliers}. Concerning the latter problem, we note that the analysis developed in \cite{CXY, Ri} cannot be directly carried out in the present situation, since the underlying states $\om$ are constructed only for very particular deformation angle $\a$, deeply depending on that angle, see \cite{M} and \cite{FS}, Section 3. Therefore, the deformation angle $\a$ cannot be merely considered as a parameter, and a more refined analysis should be used for the non type $\ty{II_1}$ cases in the sequel.

\section{preliminaries}
\label{sec2}

\subsection{Notations}

Let $E$ be a normed space. We simply denote by $\|\,{\bf\cdot}\,\|$ its norm whenever no confusion arises. In particular, $\|x\|\equiv\|x\|_\ch$ will be the Hilbertian norm of $x\in\ch$. For a continuous or measurable function $f$ defined on the locally compact space $X$ equipped with the Radon measure $\m$, $\|f\|\equiv\|f\|_\infty$ will denote the ``esssup'' norm (or the "sup" norm  for $f$ continuous) of the function $f$. 

In such a situation, $C_{\rm b}(X)$ will denote the $C^*$-algebra consisting of all bounded continuous functions defined on $X$, equipped with the natural algebraic operations, and norm 
$\|f\|\equiv\|f\|_\infty:=\sup_{x\in X}|f(x)|$. The situation of a point-set $X$, together with the $C^*$-algebra $\cb(X)=C_{\rm b}(X)$ of all bounded functions defined on $X$, is reduced to 
a particular case of the previous one by considering $X$ equipped with the discrete topology.

\vskip.3cm

Let $\bt:=\{z\in\bc\mid |z|=1\}$ be the abelian group consisting of the unit circle. The dual topological group $\widehat{\bt}$ is isomorphic to the discrete group $\bz$. The corresponding Haar measures are the normalised Lebesgue measure $\di m=\frac{\di z}{2\pi\imath z}=\frac{\di \th}{2\pi}$ for $z=e^{\imath\th}$ on the circle, and the counting measure on $\bz$, respectively. 

The Fourier transform and anti-transform of a bounded signed Radon measure $\m\in C(\bt)^*$ are respectively defined as
$$
\widehat{\m}(n):=\int_\bt z^{-n}\di\m(z)\,,\,\,\, \widecheck{\m}(n):=\int_\bt z^{n}\di\m(z),\quad n\in\bz.
$$
For $f\in L^1(\bt,m)$, its Fourier transform and anti-transform are usually defined as those of the measure $f\di m$:
$$
\widehat{f}(n):=\oint f(z)z^{-n}\frac{\di z}{2\pi\imath z},\,\,\,
\widecheck{f}(n):=\oint f(z)z^{n}\frac{\di z}{2\pi\imath z},\quad n\in\bz.
$$
The symbol $D$ denotes the derivative w.r.t. the natural argument $z\in\bt$, {\it i.e.} $D\equiv\frac{\di\,}{\di z}$, and if $z=e^{\imath\th}$ then $\frac{\di\,}{\di\th}=\imath z\frac{\di\,}{\di z}$. 
In particular, if $g\in C^1(\bt)$, 
$$
\frac{\di\,}{\di\th}\,g(e^{\imath\th})=\imath zDg(z).
$$

\vskip.3cm

Let $(X,\cb)$ be a measurable space, together with two $\s$-additive positive measures $\m$ and $\n$ on the $\s$-algebra $\cb$. If 
$\m$ dominates $\n$ in the sense of measures ({\it e.g.} \cite{R}, Section 6), we write $\n\preceq\m$. If $\n\preceq\m$ and $\m\preceq\n$, then $\m$ and $\n$ are equivalent as measures, and we write $\m\sim\n$.

We also write $X\sim Y$ whenever two sets $X$ and $Y$, equipped with the same algebraic structure, are algebraically isomorphic. For example, if $\ga$ and $\gb$ are involutive algebras (and in particular, $C^*$-algebras or $W^*$-algebras), 
we write $\ga\sim\gb$ if there exists a $*$-isomorphism $\r:\ga\to\gb$ of $\ga$ onto $\gb$.

\vskip.3cm

For a function $f$, $M_f$ denotes the multiplication operator acting on functions $g$ as $M_fg:=fg$. For example, if $f$ is a measurable function on the measure space $(X,\cb,\n)$, bounded almost everywhere, the multiplication operator $M_f$ is the closed operator acting on $L^2(X,\cb,\n)$ with domain
$$
\cd_{M_f}:=\bigg\{g\in L^2(X,\cb,\n)\mid \int_X\big|f(x)g(x)\big|^2\di\n(x)<+\infty\bigg\},
$$
defined for $g\in\cd_{M_f}$ as
$$
(M_fg)(x):=f(x)g(x),\quad x\in X.
$$

\vskip.3cm

If $\mathpzc{f}:X\to X$ is an invertible map on a point-set $X$:
\begin{itemize} 
\item[{\bf --}] $\mathpzc{f}^0:=\id_X$, and its inverse is denoted by $\mathpzc{f}^{-1}$;
\item[{\bf --}] for the $n$-times composition, $\mathpzc{f}^n:=\underbrace{\mathpzc{f}\circ\cdots\circ \mathpzc{f}}_{n-\text{times}}$;
\item[{\bf --}] for the $n$-times composition of the inverse, $\mathpzc{f}^{-n}:=\underbrace{\mathpzc{f}^{-1}\circ\cdots\circ \mathpzc{f}^{-1}}_{n-\text{times}}$.
\end{itemize}
Therefore, $\mathpzc{f}^{n}:X\to X$ is meaningful for any $n\in\bz$ with the above convention, and provides an action of the group $\bz$ on the set $X$. 

\subsection{Modular spectral triples}

Concerning the usual terminology, the main concepts and results in operator algebras theory such as the Tomita modular theory and so forth, the reader is referred to \cite{BR, St, T} and the reference cited therein. 

Let $\ga$ be a unital $C^*$-algebra. Denote by 
$\cs(\ga)\subset\ga^*_+$ the set of the states, that is the positive and normalised functionals on the $C^*$-algebra $\ga$. 

Let $\om\in\cs(\ga)$ with $s(\om)\in Z(\ga^{**})$. For such states with central support in the bidual, it can be shown ({\it e.g.} \cite{NSZ}) that the cyclic vector $\xi_\om\in\ch_\om$ is separating for $\pi_\om(\ga)''$, $(\ch_\om,\pi_\om,\xi_\om)$ being the GNS representation of $\om$. Therefore,
$$
\ch_\om\ni\pi_\om(a)\xi_\om\mapsto\pi_\om(a^*)\xi_\om\in\ch_\om
$$ 
is well defined on the dense subset $\big\{\pi_\om(a)\xi_\om\mid a\in\ga\big\}\subset\ch_\om$, and closable. The polar decomposition of its closure $S_\om$, named the Tomita involution, is usually written as $S_\om=J_\om\D_\om^{1/2}$, where $J_\om$ and $\D_\om$ are the Tomita conjugation and modular operator, respectively. 
We also report the useful relation ({\it cf.}  \cite{St}, Section 2.12)
\begin{equation}
\label{tomis}
J_\om f(\D_\om)J_\om=\bar f(\D^{-1}_\om),
\end{equation}
holding for each Borel function $f$ on $\br_+$.

Here, we provide the definition of (even) modular spectral triple we will use in the sequel, which is slightly different from the analogous one in \cite{CM} and generalise Definition 2.4 in \cite{FS}. We refer the reader to \cite{C} and the references cited therein, for general aspects and some natural applications of spectral triples.
\begin{defin}
\label{fmst}
For each $\eta\in[0,1]$, a {\it modular spectral triple} associated to a unital $C^*$-algebra $\ga$ is a triplet $(\om,\ca,L)_\eta$, where $\om\in\cs(\ga)$, $\ca\subset\ga$ is a dense $*$-algebra and $L$ is a densely defined closed operator acting on $\ch_\om$ satisfying the following conditions:
\begin{itemize}
\item[(i)] $s(\om)\in Z(\ga^{**})$;
\item[(ii)] the deformed Dirac operator 
$$
D_L^{(\eta)}:=\begin{pmatrix} 
	 0 &\D_\om^{\eta-1}L\D_\om^{-\eta}\\
	\D_\om^{-\eta}L^*\D_\om^{\eta-1}& 0\\
     \end{pmatrix} 
$$
acting on $\ch_\om\oplus\ch_\om$, uniquely defines a selfadjoint operator with compact resolvent: $D_L^{(\eta)}$ is densely defined essentially selfadjoint 
with $\left(1+\left(\overline{D_L^{(\eta)}}\right)^2\right)^{-1/2}$ compact;
\item[(iii)] for each $a\in\ca$, the deformed commutator
$$
\cd_L^{(\eta)}\big(\pi_\om(a)\big):=\imath\begin{pmatrix} 
	0&\D_\om^{\eta-1}[L,\pi_\om(a)]\D_\om^{-\eta}\\
	\D_\om^{-\eta}[L^*,\pi_\om(a)]\D_\om^{\eta-1}& 0\\
     \end{pmatrix}
$$
uniquely defines a bounded operator: $\overline{\cd_{\cd_L^{(\eta)}\left(\pi_\om(a)\right)}}=\ch_\om\oplus\ch_\om$ and 
$$
\sup\big\{\big\|\cd_L^{(\eta)}\big(\pi_\om(a)\big)\xi\big\|\mid \xi\in\cd_{\cd_L^{(\eta)}\left(\pi_\om(a)\right)},\,\,\|\xi\|\leq1\big\}<+\infty;
$$
\item[(iv)] $\pi_\om(\ca)\cd_L\subset \cd_L$, $\pi_\om(\ca)\cd_{L^*}\subset \cd_{L^*}$.
\end{itemize}
\end{defin}
\vskip.5cm
The closure of the Dirac operator in (ii) will be also denoted as $D_L^{(\eta)}$ with an abuse of notation.

Notice that, similarly to the untwisted case in order to study the possibly associated Fredholm module (see {\it e.g.} \cite{FMR}),
condition (iv) above is included only to make the commutators $[L,\pi_\om(a)]$, $[L^*,\pi_\om(a)]$ appearing in (ii) directly meaningful.

\subsection{The noncommutative 2-torus}

We fix the noncommutative 2-torus based on the deformation of the classical 2-torus $\bt^2$, corresponding to the rotation of the angle $4\pi\a$. The multiplicative factor 2 is introduced only for a pure matter of convenience explained in \cite{FS}. 
%In addition, we restrict our analysis to the case when $\a$ is irrational. 
Indeed, for a fixed $\a\in\br$, the {\it noncommutative torus} $\ba_{2\a}$ associated with the rotation by the angle $4\pi\a$, is the universal $C^*$-algebra with identity $I$ generated by the commutation relations involving two noncommutative unitary indeterminates $U,V$:
\begin{equation}
\label{ccrba}
\begin{split}
&UU^*=U^*U=I=VV^*=V^*V,\\
&UV=e^{4\pi\imath\a}VU.
\end{split}
\end{equation}
From now on, without loosing generality ({\it cf.} \cite{B}) we also assume that $\a\in(0,1/2)$. 

We express $\ba_{2\a}$ in the so-called {\it Weyl form}. Let ${\bf a}:=(m,n) \in \bz^2$ be a double sequence of integers, and define
\begin{equation}
\label{wcea}
W({\bf a}):=e^{-2\pi\imath\a mn}U^mV^n,\quad {\bf a}\in\bz^2.
\end{equation}
Obviously, $W({\bf 0})=I=\idd_{\ba_{2\a}}$, and the commutation relations \eqref{ccrba} become
\begin{align}
\label{ccrba1}
\begin{split}
&W({\bf a})W({\bf A})=W({\bf a}+{\bf A})e^{2\pi\imath\a\s({\bf a},{\bf A})},\\
&W({\bf a})^*=W(-{\bf a}),\quad {\bf a}, {\bf A}\in\bz^2,
\end{split}
\end{align}
where the symplectic form $\s$ is defined as
$$
\s({\bf a},{\bf A}):=(mN-Mn),\quad {\bf a}=(m,n), \,{\bf A}=(M,N) \in\bz^2.
$$
We now fix a function $f\in\cb(\bz^2)$, which we may assume to have finite support. The element $W(f)\in\ba_{2\a}$ is then defined as
$$
W(f):=\sum_{{\bf a}\in\bz^2}f({\bf a})W({\bf a}).
$$
The set 
$$
\{W(f)\mid f\,\text{is a function on $\bz^2$ with finite support}\}
$$ 
provides a dense $*$-algebra of $\ba_{2\a}$. Indeed, the relations \eqref{ccrba1} are transferred on the generators $W(f)$ as follows:
$$
W(f)^*=W(f^\star),\quad W(f)W(g)=W(f\star g),
$$
where
\begin{align*}
&f^\star({\bf a}):=\overline{f(-{\bf a})},\\
&(f\star g)({\bf a})=\sum_{{\bf A}\in\bz^2}f({\bf A})g({\bf a}-{\bf A})e^{-2\pi\imath\a\s({\bf a},{\bf A})}.
\end{align*}
The deformed convolution above depends on the chosen number $\a$. Since such a number is fixed in the sequel, we simply denote that with the $\star$ symbol.

For irrational numbers $\a$, we also recall that $\ba_{2\a}$ is simple and has a necessarily unique and faithful trace
$$
\t(W(f)):=f({\bf 0}),\quad W(f)\in\ba_{2\a}.
$$
In this situation, by \cite{B}, Remark 1.7, any element $A\in\ba_{2\a}$ is uniquely determined by its {\it Fourier coefficients}
\begin{equation}
\label{fzocf}
f({\bf a}):=\t(W(-{\bf a})A),\quad {\bf a}\in\bz^2.
\end{equation}
We note that, for each fixed $n\in\bz$, $f^{(n)}(m):=f(m,n)$ defines a sequence whose Fourier anti-transforms 
$$
\widecheck{f^{(n)}}(z):=\sum_{m\in\bz}f(m,n) z^m,\quad z\in\bt,
$$
provides a sequence of continuous functions $\left\{\widecheck{f^{(n)}}\mid n\in\bz\right\}\subset C(\bt)$.

In order to obtain explicit type $\ty{III}$ representations of the noncommutative 2-torus, we tacitly suppose that the deformation angle $\a$ is irrational, satisfying sufficiently fast approximation properties by rationals, see below.

\subsection{Diffeomorphisms of the unit circle}
\label{2dixf}

The Liouville numbers, denoted here as {\bf L}-numbers, are necessarily irrational, and satisfy by definition the fast approximation by rationals as follows.
\begin{itemize}
\item[{\bf (L)}] A {\it Liouville number} $\a\in(0,1)$ is a real number such that, for each $N\in\bn$ the inequality 
\begin{equation*}
%\label{uLn}
\left|\a - \frac{p}{q}\right|< \frac{1}{q^N}
\end{equation*}
has an infinite number of solutions for $p,q\in\bn$ with ${\rm gcd}(p,q)=1$.
\end{itemize}
We also consider numbers satisfying the following faster approximation by rationals, and denoted as  {\bf UL}-numbers.
\begin{itemize}
\item[{\bf (UL)}] A  {\it Ultra-Liouville number} $\a\in(0,1)$ is a real number such that, for each $\l>1$ and $N\in\bn$, the inequality
\begin{equation*}
%\label{uLn1}
\left|\a - \frac{p}{q}\right|<\frac{1}{\l^{q^N}}
\end{equation*}
again admits an infinite number of solutions for $p,q\in\bn$ with  ${\rm gcd}(p,q)=1$.
\end{itemize}
We refer the reader to \cite{Ki, S} for properties and details on such Liouville numbers. 

In order to construct states on $\ba_{2\a}$ with a nontrivial modular structure, and then type $\ty{III}$ representations with associated nontrivial modular spectral triples, we consider an orientation-preserving $C^\infty$-diffeo\-morphism (called simply "a diffeomorphism") $\mathpzc{f}\in C^\infty(\bt)$ with the rotation number ({\it e.g.} \cite{KH}) $\r(\mathpzc{f})=2\a$. By the theorem of Denjoy, $\mathpzc{f}$ is conjugate to the rotation $R_{2\a}$ of the angle $4\pi\a$ through an uniquely determined homeomorphism 
$\mathpzc{h}_\mathpzc{f}$ of the unit circle satisfying $\mathpzc{h}_\mathpzc{f}(1)=1$, and
\begin{equation}
\label{dehcm}
\mathpzc{f}=\mathpzc{h}_\mathpzc{f}\circ R_{2\a}\circ\mathpzc{h}_\mathpzc{f}^{-1}.
\end{equation}
In such a situation, $T$ is defined as the "square root" of $\mathpzc{f}$:
\begin{equation}
\label{sqfzc}
T:=\mathpzc{h}_\mathpzc{f}\circ R_{\a}\circ\mathpzc{h}_\mathpzc{f}^{-1}.
\end{equation}
We note that
$$
\m_\mathpzc{f}:=(\mathpzc{h}_\mathpzc{f})^*m=m\circ \mathpzc{h}_\mathpzc{f}^{-1},
$$
is the unique invariant measure, which is also ergodic, for the natural action of $\mathpzc{f}$ on $\bt$. For a Diophantine number $\a$, it can be shown that $\mathpzc{h}_\mathpzc{f}$ is indeed smooth, and thus $\m_\mathpzc{f}\sim m$, the last denoting the normalised Haar measure on the circle.

For a Liouville number $\a$, things are quite different. In fact, there are diffeomorphisms as above for which the unique invariant measure $(\mathpzc{h}_\mathpzc{f})^*m$ is singular w.r.t. the Haar measure $m$: $(\mathpzc{h}_\mathpzc{f})^*m\perp m$. 

Summarising, if $\a$ is Diophantine the construction in \cite{FS} produces representations which are equivalent to the tracial one, leading to type $\ty{II_1}$ von Neumann factors. Conversely, if $\a$ is a {\bf L}-number, in \cite{M} diffeomorphisms as above were constructed with the prescription that
\begin{equation*}
%\label{sile}
\pi_\om(\ba_{2\a})''\sim L^\infty(\bt,\di\th/2\pi)\ltimes_{\b}\bz
\end{equation*}
is a type $\ty{II_\infty}$, or $\ty{III_\l}$ von Neumann factor for each $\l\in[0,1]$. Here, $\b$ is the action induced on functions $g$ on the unit circle by the corresponding action of $\mathpzc{f}$ on points: $\b(g)=g\circ\mathpzc{f}^{-1}$.
If in addition $\a$ is a {\bf UL}-number, then it is also possible to explicitly construct nontrivial modular spectral triples for $\ba_{2\a}$, which fulfil the prescribed properties. The reader is referred to \cite{C, CM} for the significance of the concept of spectral triples and the modular ones in noncommutative geometry, and ti \cite{FS} for the construction of nontrivial 
({\it i.e.} non type $\ty{II_1}$) modular spectral triples.

Another natural object associated to $C^1$-diffeomorphisms of the unit circle which plays a crucial role in constructing modular spectral triples, is the so-called {\it growth sequence}
$\{\G_n(\mathpzc{f})\mid n\in\bn\}$, defined as
$$
\G_n(\mathpzc{f}):=\max\{\|D\mathpzc{f}^n\|_\infty,\|D\mathpzc{f}^{-n}\|_\infty\},\quad n\in\bn.
$$

For smooth diffeomorphisms $\mathpzc{f}$ of the circle conjugate to an irrational rotation, it can be shown that $\G_n(\mathpzc{f})$ cannot have a too wild behaviour at infinity, see \cite{W}. If $\a$ is Diophantine, then $\G_n(\mathpzc{f})$ is bounded because $\mathpzc{h}_\mathpzc{f}$ in \eqref{dehcm} is smooth. For diffeomorphisms constructed in \cite{M} for a {\bf L}-number $\a$, it is shown in \cite{FS} that $\G_n(\mathpzc{f})=o(n)$, and  
$\G_n(\mathpzc{f})=o(\ln n)$ if in addition $\a$ is a {\bf UL}-number.

Consider $g\in\cb(\bz^2)$ such that $W(g)\in\ba_{2\a}$. For $k\in\bn$, $l=0,1$, define the following sequence of seminorms
\begin{equation*}
%\label{semi}
\r_{k,l}\big(W(g)\big):=\sup_{n\in\bz}\left\{(|n|+1)^k\left\|D^l\left(\widecheck{g^{(n)}}\circ R^{-n}\circ \mathpzc{h}_{\mathpzc{f}}^{-1}\right)\right\|_\infty\right\},
\end{equation*}
provided $\widecheck{g^{(n)}}\circ R^{-n}\circ \mathpzc{h}_{\mathpzc{f}}^{-1}\in C^1(\bt)$, $n\in\bz$. Define
$$
\ba^{o}_{2\a}:=\bigg\{W(g)\mid\r_{k,l}\big(W(g)\big)<+\infty,\,\, k\in\bn,\, l=0,1\bigg\},
$$
which was shown in \cite{FS}, Section 11, to be a unital $*$-subalgebra of $\ba_{2\a}$, stable under the entire functional calculus.

\subsection{States on the noncommutative 2-torus}

As explained before, for $\a\in(0,1/2)$ being an irrational number, fix a diffeomorphism $\mathpzc{f}\in C^\infty(\bt)$ which we suppose to be always $C^\infty$ without any specific mention, satisfying \eqref{dehcm}. As previously pointed out, if $\a$ is a {\bf L}-number or a {\bf UL}-number, it will be possible to construct among others, type $\ty{III}$ representations and in addition nontrivial modular spectral triples, respectively.

The starting point is to consider the unique homeomorphism $\mathpzc{h}_\mathpzc{f}$ satisfying \eqref{dehcm}, together with the measure $\m:=m\circ\mathpzc{h}_\mathpzc{f}$. Since 
\begin{equation}
\label{ommu1}
\m\sim\m\circ R_{2\a}^{n}\,,\quad n\in\bz,
\end{equation}
it is shown in \cite{FS} that the state $\om\equiv\om_\m$ on $\ba_{2\a}$ defined as
\begin{equation*}
%\label{ommu}
\om(W(f)):=\sum_{m\in\bz}\widecheck{\m}(m)f(m,0)
\end{equation*}
has central support in the bidual $\ba_{2\a}^{**}$: $s(\om)\in Z(\ba_{2\a}^{**})$. The GNS representation $(\ch_\om,\pi_\om,\xi_\om)$ associated to $\om$ is given by
\begin{equation*}
%\label{ggnnss666}
\begin{split}
\ch_{\om}&=\ell^2\big(\bz;L^2(\bt, m)\big),\\
(\pi_{\om}(W(f))g)_n(z)&=\sum_{l\in\bz}\left(\widecheck{f^{(l)}}\circ \mathpzc{h}^{-1}_{T^2}\circ T^{2n-l}\right)(z)g_{n-l}(z),\\
(\xi_{\om})_n(z)&=\d_{n,0},\quad z\in\bt,\,\,n\in\bz,
\end{split}
\end{equation*}
where $T$ is a "square root" of $\mathpzc{f}$ given in \eqref{sqfzc}.

By \eqref{ommu1}, we have 
$$
m\sim m\circ \mathpzc{f}^{n},\quad n\in\bz
$$
(which of course holds true even for each $C^1$-diffeomorphism),
with Radon-Nikodym derivative
\begin{equation}
\label{rcanz}
\d_n(z):=\frac{\di m\circ\mathpzc{f}^{n}}{\di m}(z)=\frac{z(D\mathpzc{f}^{n})(z)}{\mathpzc{f}^{n}(z)},\quad z\in\bt,\,n\in\bz.
\end{equation}
The modular structure is then expressed as follows. With $a\in\br$, on 
$$
\cd_{\D^{a/2}_\om}:=\bigg\{x\in\ch_\om\mid \sum_{n\in\bz}\int_\bt\d_n(z)^a|x_n(z)|^2\di m(z)<+\infty\bigg\}
$$
we have for the modular operator:
\begin{equation*}
(\D^{a/2}_\om x)_n(z)=\d_n(z)^{a/2}x_n(z),\quad n\in\bz,\,\,z\in\bt.
\end{equation*}
Concerning the modular conjugation, we get
$$
(J_\om x)_n(z)=\d_n(z)^{1/2}\overline{x_{-n}(\mathpzc{f}^n(z))},\quad n\in\bz,\,\,z\in\bt,
$$
which is meaningful for each $x\in\ch_\om$.

\subsection{Operator spaces}
\label{ops}

We consider a normed space $E$ equipped with  a sequence of norms on $\bm_n(E)$, the spaces of the $n\times n$ matrices with entries in $E$, satisfying for $x\in\bm_n(E)$, $y\in\bm_m(E)$, $a,b\in\bm_n$, $n, m=1,2,\dots$, the following properties:
\begin{itemize}
\item[(i)] $\|axb\|_{\bm_n(E)}\leq\|a\|_{\bm_n}\|x\|_{\bm_n(E)}\|b\|_{\bm_n}$,
\item[(ii)] $\|x\oplus y\|_{\bm_{n+m}(E)}=\max\{\|x\|_{\bm_{n}(E)},\|y\|_{\bm_{m}(E)}\}$.
\end{itemize}
Such a normed space, together with these sequence of norms, is said to be an (abstract) {\it operator space}. Conversely, for a subspace $E\subset\ga$ of a $C^*$-algebra $\ga$, the norms on $\bm_n(E)$ inherited by the inclusions $\bm_n(E)\subset\bm_n(\ga)$,
automatically satisfy (i) and (ii) above, and therefore $E$ is a (concrete) operator space in a natural way. It is proved in \cite{Ru} that a normed space $E$, equipped with a sequence of norms on all $\bm_n(E)$, can be faithfully represented as a subspace of a $C^*$-algebra if and only if these norms satisfy (i) and (ii) above.

Let $E$ be an operator space. The topological dual $E^*$ can be viewed also as an operator space, when $\bm_n(E^*)$ are equipped with the norms given for $f\in\bm_n(E^*)$ by
$$
\|f\|_{\bm_n(E^*)}:=\sup\left\{\|f(a)\|_{\bm_{nm}}\mid \|a\|_{\bm_m(E)}\leq1, m\in\bn\right\}.
$$
Here, with $a\in \bm_m(E)$, the matrix $f(a)$ is defined as
$$
f(a)_{(i,j)(k,l)}:=f_{ik}(a_{jl}),\quad i,k=1,\dots,n,\,\,j,l=1,\dots,m.
$$
Such an operator space $E^*$ is called in \cite{Bl} {\it the standard dual} of $E$.

Let $E,F$ be operator spaces. A  linear map $T:E\to F$ is said to be {\it completely bounded} if for $T\otimes\id_{\bm_n}:\bm_n(E)\to\bm_n(F)$ we have
$$
\|T\|_{cb}:=\sup_n\big\|T\otimes\id_{\bm_n}\big\|<+\infty.
$$
Examples of linear maps which are bounded but not completely bounded, can be easily exhibited, see {\it e.g.} Section \ref{sec2nclp}.

\section{on noncommutative $L^p$-spaces}
\label{sec2nclp}

\noindent

We fix a normal faithful state $\om$ on the $W^*$-algebra $M$. The $*$-algebra of {\it analytic elements} for the action of the modular group, that is that formed by all elements for which
$\br\ni t\mapsto \s^\om_{t}(a)\in M$
has a analytic continuation to an entire function $\bc\ni z\mapsto\s^\om_{z}(a)\in M$ (see {\it e.g.} \cite{St}) is denoted by $M_\infty$. 

For $\th\in[0,1]$, consider the noncommutative $L^p$-spaces $L^p_\th(M,\om)=L^p_\th(M)$, as $\om$ is kept fixed. These can be constructed by complex interpolation by considering various embeddings $M\hookrightarrow M_*$, one for each $\th$, see \cite{F, K, Te}. Here, $\th=0,1$ correspond to the left and right embeddings
$$
M\ni x\mapsto\left\{\begin{array}{ll}
                      \!\!\!\!\!\!&\om(\,{\bf\cdot}\,x)\\
                     \!\!\!\!\!\!&\om(x\,{\bf\cdot}\,)
                    \end{array}
                    \right.\!\in M_*,
$$
respectively. We also mention another equivalent way of approaching $L^p$-spaces developed in \cite{Ha}.

The key-point for the construction of noncommutative $L^p$-spaces is the modular theory ({\it e.g.} \cite{St}), and complex interpolation ({\it e.g.} \cite{BL}). In fact, for $x\in M$ the map
$$
\br\ni t\mapsto \s^\om_t(x)\om=\om\big(\,{\bf\cdot}\, \s^\om_t(x)\big)\in M_*
$$ 
extends to a bounded and continuous map on the strip $\{z\in\bc\mid-1\leq {\rm Im}(z)\leq 0\}$, analytic in the interior. After putting $L^1_\th(M):=M_*$ for each $\th\in[0,1]$,  $L^\infty_\th(M):=\iota^\th_{\infty,1}(M)\sim M$ with the norm inherited by that of $M$, and checking that the complex interpolation functor $C^\th$ based on such analytic functions coincides (equal norms) with the standard one $C_\th$ (see Theorem 1.5 in \cite{K}), it is possible to define 
$$
L^p_\th(M):=C^{1/p}\big(\iota^\th_{\infty,1}(M),M_*\big)=C_{1/p}\big(\iota^\th_{\infty,1}(M),M_*\big),\quad 1<p<+\infty.
$$
In such a way, for $1\leq p\leq q\leq+\infty$ there are contractive embeddings
$$
\iota^\th_{q,p}:L^q_\th(M)\hookrightarrow L^p_\th(M).
$$
We also remark ({\it cf.} \cite{F}) that one can equip the $ L^p$-spaces with a natural operator space structure arising from the canonical embeddings 
$$
\iota^\th_{\infty,1}\otimes\id:\bm_n(M)\to\bm_n(M_*).
$$
As stated in Proposition 3 in \cite{F}, $\iota^\th_{\infty,1}$ is indeed completely bounded. For the convenience of the reader, we provide some details about the proof of such a result omitted in the original one.
\begin{prop}
\label{ohcb}
For each $\th\in[0,1]$, the embedding 
\begin{equation}
\label{em10}
M\ni x\mapsto\iota^\th_{\infty,1}(x)=\om\big(\,{\bf\cdot}\, \s^\om_{-\imath\th}(x)\big)\in M_*
\end{equation}
is a complete contraction.
\end{prop}
\begin{proof}
By the Phragm\'en-Lindel\"of Theorem ({\it e.g.} \cite{R}, Theorem 12.8), complete boundedness for the left and right embeddings ({\it i.e.} for $\th=0,1$) implies complete boundedness for all 
$\th\in[0,1]$. Therefore, we reduce the matter to the left and right embeddings.

By considering the natural embedding $\bm_m(M)\ni y\mapsto y\oplus 0\in\bm_n(M)$ for $n\geq m$, we can restrict to the cases where $m=n$. Thus, in order to check the complete boundedness, for each $x,y\in\bm_n(M)$ and $i,j,k,l=1,\dots,n$ we should compute the operator norm of the numerical matrices $L,R\in\bm_{n^2}$ whose entries are given by
$$
L_{(ij),(kl)}:=\om(y_{jl}x_{ik}),\quad R_{(ij),(kl)}:=\om(x_{ik}y_{jl}).
$$
Concerning the second one, we easily get
$$
R=\big(\om\otimes\id_{\bm_{n^2}}\big)(x\otimes y),
$$
and therefore 
$$
\|R\|_{\bm_{n^2}}\leq\|x\|_{\bm_n(M)}\|y\|_{\bm_n(M)}.
$$
For the first one, we can straightforwardly see that there exist a permutation (depending on $n$) $\s_n$ of the set 
$$
(11),\dots, (1n),\dots,(n1),\dots, (nn)\sim 1,\dots,n^2,
$$
and therefore a selfadjoint unitary $V_n$ acting on $\bc^{n^2}$ such that 
$$
L=V_n\big(\om\otimes\id_{\bm_{n^2}}\big)(y\otimes x)V_n.
$$
For example, when $n=2$, 
$$
\s_2=\begin{pmatrix} 
	 1\,\,2\,\,3\,\,4\\
	1\,\,3\,\,2\,\,4\\
     \end{pmatrix},
$$
and for $n=3$, 
$$
\s_3=\begin{pmatrix} 
	 1\,\,2\,\,3\,\,4\,\,5\,\,6\,\,7\,\,8\,\,9\\
	1\,\,4\,\,7\,\,2\,\,5\,\,8\,\,3\,\,6\,\,9\\
     \end{pmatrix}.
$$
Consequently, again
$$
\|L\|_{\bm_{n^2}}\leq\|x\|_{\bm_n(M)}\|y\|_{\bm_n(M)}.
$$
\end{proof}
It is possible to show that such an operator space structure coincides with the Pisier OH-structure ({\it cf.} \cite{Pi})
at level of Hilbert spaces, that is 
$$
C_{1/2}\big(\iota^\th_{\infty,1}(M),M_*\big)\equiv L^2(M)={\rm OH}(\ch_\om),\quad \th\in[0,1],
$$
see  \cite{F}, Theorem 5. 

With such identifications, we have that the topological dual $L^p_\th(M)'$ is completely isomorphic to $L^q_{1-\th}(M)$ in a canonical way, see \cite{K, F}.

Such an operator space structure on the noncommutative $L^p$-spaces obtained by interpolation arising by the various embeddings of $M\equiv L^\infty(M)$ into $M_*\equiv L^1(M)$ as before, is called {\it canonical}.

To simplify the notations, we put for $p,q\in[1,+\infty]$ with $q\geq p$,
$$
i_{q,p}:=\iota_{q,p}^0,\quad j_{q,p}:=\iota_{p,q}^1.
$$
For $a\in M$, we denote with $L_a$ and $R_a$ its left and right embedding in $M_*$, that is
\begin{equation*}
%\label{efcm}
L_a:=\iota_{\infty,1}^0(a)=i_{\infty,1}(a),\quad R_a:=\iota_{\infty,1}^1(a)=j_{\infty,1}(a).
\end{equation*}
We note that it is also customary to equip $M_*$ with the operator space structure arising from the various embeddings relative to $\th\in[0,1]$ of $M$ into the predual of the opposite algebra 
$(M^\circ)_*$, see {\it e.g.} Section 2.1 of \cite{CXY}. Denote by $\k^\th_{\infty,1}:M\hookrightarrow (M^\circ)_*$ such an embedding, and consider the canonical identification
$o:M_*\to (M^\circ)_*$
%After defining
%$$
%M_*\ni\f\mapsto o(\f)\in (M^\circ)_*
%$$
given by 
$$
o(\f)(x^\circ):=\f(x),\quad \f\in M_*,\,\,x^\circ\in M^\circ.
$$ 
Apart from other things, this choice means to exchange the left with the right embedding as follows: $\k^\th_{\infty,1}=o\circ\iota^{1-\th}_{\infty,1}$, see {\it e.g.} \cite{K}, Proposition 7.3, and   \cite{F}, Proposition 3. With this choice, the embeddings $\k^\th_{\infty,1}$ would not be completely bounded. To this aim, we prove the following, probably known to the experts, 
\begin{prop}
For the Hilbert space $\ch=\ell^2(\bn)$, the map $(o^t)^{-1}:\cb(\ch)\to\cb(\ch)^\circ$ is not completely bounded.
\end{prop}
\begin{proof}
First we note that, in general, $(o^t)^{-1}(x)=x^\circ$, provides the canonical identification between $M$ and $M^\circ$. Let now $\t$ be the normalised ({\it i.e.} taking value 1 on minimal projections) normal, semifinite, faithful trace on $\cb(\ch)$. It is easily seen that $\cb(\ch)^\circ$ is $*$-isomorphic to
$\pi_\t(\cb(\ch))'$, where the isomorphism is given by
$$
T(x^\circ)=J_\t\pi_\t\big(o^t(x^\circ)\big)^*J_\t.
$$
With $j:=T\circ(o^t)^{-1}:\cb(\ch)\to\pi_\t(\cb(\ch))'$ given by
$j(x)=J_\t\pi_\t(x)^*J_\t$, $(o^t)^{-1}$ would be completely bounded if and only if $j$ were so. Consider now the transpose map ({\it cf.} \cite{To1, To2})
$\th:\cb(\ch)\to\cb(\ch)$ on numerical matrices. On $\ch_\t=L^2(\ch)$ corresponding to all Hilbert-Schmidt class operators acting on $\ch=\ell^2(\bn)$, we define the unitary selfadjoint operator
$$
Vx:=\th(x), \quad x\in\ch_\t.
$$
With $I_n$ the identity operator on $\bc^n$, it is then straightforward to check that
$$
\big(j\otimes\id_{\bm_n}\big)(x)=(V\otimes I_n)\big(\pi_\t\circ\th\otimes\id_{\bm_n}\big)(x)\big(V\otimes I_n),
$$
and therefore by \cite{To1}, Theorem 1.2,
$$
\big\|\big(j\otimes\id_{\bm_n}\big)(x)\big\|_{\bm_n(\cb(\ch_\t))}=\big\|\big(\th\otimes\id_{\bm_n}\big)(x)\big\|_{\bm_n(\cb(\ch))}=n.
$$
Hence, $j$ cannot be completely bounded.
\end{proof}

\section{the fourier transform for the noncommutative torus}
\label{sec3}

\noindent

For each fixed $\om\in\cs(\ba_{2\a})$ put $M:=\pi_\om(\ba_{2\a})''$, the von Neumann algebra generated by the GNS representation $\pi_\om$. 

We are interested in states $\om$ with central support in the bidual $\ba_{2\a}^{**}$, that is the cyclic vector $\xi_\om$ is also separating for $M$.

Concerning a suitable definition of the Fourier transform for the cases when the reference measure is not the trace, we have to understand what is the right replacement $w_{kl}\in\ba_{2\a}$ of the "characters" $(e^{\imath\th_1})^{k}(e^{\imath\th_2})^{l}\in C(\bt^2)$ in the commutative case, where for the tracial case correspond to $U^{k}V^{l}$. Another natural question is how to "integrate" w.r.t. the reference state. Indeed, for 
$x\in M$ and $i_{\infty,1}(x)$, $j_{\infty,1}(x)$ the left and right embedding respectively, for the definition of the Fourier coefficients we have at least the following alternatives corresponding to the left and the right embeddings of $M$ into $M_*$:
\begin{align*}
\om(w_{kl}^*x)=&\langle\pi_\om(x)\xi_\om,\pi_\om(w_{kl})\xi_\om\rangle,\\
\om(w_{kl}\s^\om_{-\imath}(x))=&\langle\pi_\om(w_{kl})\xi_\om,\pi_\om(x^*)\xi_\om\rangle.
\end{align*}
In the second alternative, the choice not to consider the adjoint of the chosen characters is only a matter of convenience, see Section \ref{altfou}.
Such a choice of the characters $w_{kl}$ may depend on the selected alternative.

Notice also that for the tracial cases and even for the commutative case of $C(\bt^2)$, these alternatives collapse to only one due to the tracial property of the underlying "measure". 
In the cases we are managing in the present section, one possible choice is reported as follows.

For $T^2=\mathpzc{f}$ as in \eqref{sqfzc}, put $f_k(m,n):=\d_{m,0}\d_{n,k}$, $g_l(m,n):=\widehat{(h_{T^2})^l}(m)\d_{n,0}$, and define
$\{u_{kl}\mid k,l\in\bz\}\subset\pi_\om(\ba_{2\a})$, where
$$
u_{kl}:=\pi_\om\big(W(f_k)W(g_l)\big),\quad k,l\in\bz.
$$
We can embed such elements in $\ch_\om=\bigoplus_\bz L^2(\bt,\di m)$ by defining
$e^{kl}:=u_{kl}\xi_\om$, obtaining
$$
e^{kl}_n(z)=z^l\d_{n,k},\quad z\in\bt,\,\,k,l\in\bz.
$$
It is easily seen that
\begin{equation}
\label{ortho}
\langle e^{kl}, e^{rs}\rangle=\d_{k,r}\oint z^{l-s}\frac{\di z}{2\pi\imath z}=\d_{k,r}\d_{l,s},
\end{equation}
that is $\{e^{kl}\mid k,l\in\bz\}\subset\ch_\om\equiv\bigoplus_\bz L^2(\bt,m)$ forms an orthonormal system
which is indeed a basis.

For $\om\in\cs(\ba_{2\a})$, we denote with an abuse of notation, again with $\om$ its vector state extension to all of $M=\pi_\om(\ba_{2\a})''$:
$$
\om(a):=\om_{\xi_\om}(a)=\langle a\xi_\om,\xi_\om\rangle,\quad a\in \pi_\om(\ba_{2\a})''.
$$
In addition, to make uniform the notation we write $L^1_0(M)=M_*$, even if $L^1_\th(M)=M_*$, independently on $\th\in[0,1]$.
\begin{defin}
For $x\in M_*=L_0^1(M)$ of the form $L_a$, we define the sequence $\widehat{x}(k,l)$ of the Fourier coefficients as
\begin{equation}
\label{for1}
\widehat{x}(k,l):=\om(u_{kl}^*a),\quad k,l\in\bz.
\end{equation}
\end{defin}
Since 
\begin{equation}
\label{forelza1}
|\om(u_{kl}^*a)|\leq \|L_a\|=\|x\|,\quad k,l\in\bz,
\end{equation}
\eqref{for1} extends by continuity to all $L_0^1(M)$, providing a sequence $\widehat{x}\in\ell^\infty(\bz^2)$.

We have the following noncommutative version of the Riemann-Lebesgue Lemma.
\begin{thm}
\label{rlcz}
We have $\widehat{L_0^1(M)}\subset c_0(\bz^2)$ with $\|\widehat{x}\|_{c_0}\leq\|x\|_{L^1}$. In addition, $\widehat{x}=0\Rightarrow x=0$.
\end{thm}
\begin{proof}
By the definition of the Fourier coefficients, if $x\in L_0^1(M)$ we see that the double sequence $\widehat{x}$ is uniformly bounded with 
$\|\widehat{x}\|_{\infty}\leq\|x\|$. 

Suppose now that $\widehat{x}=0$. This means that $x(y)=0$ on the set made of elements $y$ such that $y^*$ is of the form
$$
y^*=\sum_{|r|,|s|\leq n} a_{rs}u_{rs},
$$
which generate a norm-dense subset of $\pi_\om(\ba_{2\a})$, and a $*$-weakly dense subset of 
$\pi_\om(\ba_{2\a})''\equiv M$. Since $L_0^1(M)\equiv M_*$, by continuity we conclude that $x(y)=0$ for all $y\in M$. Then $x=0$.

It remains to check that
$$
\lim_{|k|,|l|\to+\infty}\widehat{x}(k,l)=0.
$$
In fact, for the total set $\{L_{u_{rs}}\mid r,s\in\bz\}$, by \eqref{ortho},
$$
\widehat{L_{u_{rs}}}(k,l)=\d_{k,r}\d_{l,s}\to 0,
$$
as $|k|,|l|\to+\infty$.

For generic $x\in L_0^1(M)$ and $\eps>0$, we choose a finite linear combinations of the $u_{rs}$ of the form
\begin{equation*}
%\label{rcalz}
x_\eps=\sum_{|r|,|s|\leq n}c_{r,s}L_{u_{rs}},
\end{equation*}
such that $\|x-x_\eps\|<\eps$.
Then
\begin{align*}
&|\widehat{x}(k,l)|\leq|\widehat{x}(k,l)-\widehat{x_\eps}(k,l)|+|\widehat{x_\eps}(k,l)|\\
\leq&\big|\om(u^*_{kl}(x-x_\eps)\big|+|\widehat{x_\eps}(k,l)|
\leq\eps+|\widehat{x_\eps}(k,l)|\to\eps,
\end{align*}
because $\widehat{x_\eps}(k,l)\to0$ as $|k|,|l|\to+\infty$ as noticed before.

The proof follows as $\eps>0$ is arbitrary.
\end{proof}
\begin{rem}
Also for this case, it is not hard to show that $\widehat{L^1_0(M)}\subsetneq c_0(\bz^2)$. 
\end{rem}
For such a purpose, it is enough to consider the sequence of Dirichlet kernels $\{D_n\mid n\in\bn\}\subset L^1(\bt,m)$, and the associated functionals $\{X_n\mid n\in\bn\}\subset L^1(M)$ given by
$$
X_n\big(\pi_\om(W(f))\big):=\int_\bt \big(\widecheck{f^{(0)}}\circ h_{T^2}^{-1}\big)(z)D_n(z)\di m(z).
$$
Denoting by $\widehat{D_n}$ the usual Fourier transforms of the $D_n$ with an abuse of the notation, the Fourier coefficients and the $L^1$-norms of the $X_n$ are given by
$$
\widehat{X_n}(k,l)=\widehat{D_n}(l)\d_{k,0}
$$
and $\|X_n\|_{L^1}=\|D_n\|_1$, respectively. Reasoning as in \cite{R}, we have $\|\widehat{X_n}\|_\infty=1$ for each $n\in\bn$, but conversely $\|X_n\|_{L^1}\to+\infty$. This leads to a contradiction if we suppose 
that $\widehat{L^1_0(M)}=c_0(\bz^2)$.

Suppose now $\xi\in\ch_\om=L^2_0(M)$. Then its left embedding $x_\xi$ in $M_*\equiv L^1_0(M)$ is given by
$$
x_\xi(\,{\bf\cdot}\,)=\langle\,{\bf\cdot}\,\xi,\xi_\om\rangle,
$$
and therefore its Fourier transform is given for $k,l\in\bz$, by
\begin{equation}
\label{for201}
\widehat{x_\xi}(k,l)=x_\xi(u^*_{kl})=\langle u^*_{kl}\xi,\xi_\om\rangle=\langle \xi, u_{kl}\xi_\om\rangle=\langle \xi,e^{kl}\rangle.
\end{equation}
Namely, the Fourier coefficients of $x_\xi$ are precisely those arising from the orthogonal expansion of $\xi$ w.r.t. the orthonormal basis $\{e^{kl}\mid k,l\in\bz\}$ of $\ch_\om$. Therefore,
\begin{equation}
\label{for2}
\|\widehat{x_\xi}\|_{\ell^2}=\|\xi\|_{L^2},
\end{equation}
that is the Fourier transform provides a unitary isomorphism when restricted to $L_0^2(M)$ as expected. 

For the Fourier transform of elements in $L_0^2(M)$, we simply write
\begin{equation}
\label{for2a}
\widehat{\xi}(k,l):=\widehat{x_\xi}(k,l)= \langle \xi,e^{kl}\rangle,\quad \xi\in \ch_\om,\,\, k,l\in\bz.
\end{equation}
We now pass to consider the extension of the Fourier transform to $L^p$, $p\in(1,2)$. Indeed, by \eqref{forelza1} which holds true for generic element in $L_0^1(M)$ ({\it cf.} Theorem \ref{rlcz}), and 
\eqref{for2}, we can define the Fourier transform for $p\in[1,2]$ and $q$ its conjugate exponent ({\it i.e.} $1/p+1/q=1$ with the convention $1/0=\infty$) by complex interpolation ({\it cf.} \cite{K}), 
$$
\cf_{p,q}: L^p_0(M)\to\ell^q(\bz^2),
$$
which satisfies 
$$
\|\cf\|_{p,q}:=\|\cf_{p,q}\|_{\cb(L^p,\ell^q)}\leq 1,\quad p\in[1,2],
$$
because of
$$
\|\cf\|_{1,\infty}\leq 1,\quad \|\cf\|_{2,2}= 1.
$$
Concerning the involved completely bounded norms, we also have
\begin{prop}
\label{cclp}
The Fourier transforms $\cf$, given in \eqref{for1} for (a norm dense set in) $L_0^1(M)$ and in \eqref{for2a} for $L_0^2(M)$, provide complete contractions w.r.t. the canonical operator space structure for the noncommutative $L^p$-spaces obtained by interpolation.
\end{prop}
\begin{proof}
The case $p=1$ proceeds as follows. For any element $f\in\bm_n(\ell^1(\bz^2))$, its operator norm is computed in the following way.
As in Section \ref{ops}, with $a\in\bm_m(\ell^\infty(\bz^2))$ and after defining $f(a)\in\bm_{nm}$ by
$$
f(a)_{(i,j)(k,l)}:=f_{ik}(a_{jl}),\quad i,k=1,\dots,n,\,\,j,l=1,\dots,m,
$$
we have
$$
\|f\|_{\bm_n(\ell^1(\bz^2))}:=\sup\left\{\|f(a)\|_{\bm_{nm}}\mid \|a\|_{\bm_m(\ell^\infty(\bz^2))}\leq1, m\in\bn\right\}.
$$
Firstly, we note that such an element $f$ uniquely defines a linear map on $M\overline{\otimes}\ell^\infty(\bz^2)$ given by $F:=\id_M\otimes f$, assuming the form
$$
F(a):=\sum_\a f_{ik}(\a)a_\a,
$$
where we have denoted by $\a\in\bz^2$ a generic element of $\bz^2$. This provides a bounded linear map
$$
F:M\overline{\otimes}\ell^\infty(\bz^2)\rightarrow\bm_n(M)
$$
with norm $\|F\|\leq\|f\|_{\bm_n(\ell^1(\bz^2))}$. In order to check its boundedness, we proceed as follows. 

For elements $\{\xi(i),\eta(i)\}_{i=1}^n\subset\ch_\om$, fix any orthonormal basis $\{e_r\}_{1\leq r\leq 2n}$ of 
any $2n$-dimensional space $\ck\subset\ch_\om$ containing the vectors $\xi(i)$, $\eta(i)$, $i=1,\dots,n$. Take  
$a\in M\overline{\otimes}\ell^\infty(\bz^2)$. With $a_{rs}(\a):=\langle a_\a e_r,e_s\rangle$, the matrix $\big(a_{rs}(\a)\big)_{\a\in\bz^2,\,1\leq r,s\leq 2n}$ defines an element 
$\tilde a\in\bm_{2n}(\ell^\infty(\bz^2))$ with norm less than that of $a$. As usual, we also denote with $f(a)$ the $2n^2\times 2n^2$ matrix with entries given by
$$
f(a)_{(ir),(js)}:=\sum_{\a\in\bz^2}f_{ij}(\a)(a_{rs}(\a)),\quad r,s\in\bn,\,\, i,j=1,\dots, n.
$$
The elements $\{\xi(i),\eta(i)\}_{i=1}^n\subset\ch_\om$ with components $\xi_r(i):=\langle\xi(i),e_r\rangle$, $\eta_r(i):=\langle\eta(i),e_r\rangle$ w.r.t. the previous orthonormal basis, define vectors
$\tilde\xi, \tilde\eta\in\bc^n\otimes \bc^{2n}\sim\bc^{2n^2}$ with components 
$$
\tilde\xi_{ir}:=\xi_r(i),\,\, \tilde\eta_{ir}:=\eta_r(i),\quad 1\leq r\leq n,\,\, i=1,\dots, n.
$$
We compute
\begin{align*}
&\bigg|\sum_{i,j=1}^n\langle F(a)\xi_i,\eta_j\rangle\bigg|
=\bigg|\sum_{i,j,\a}f_{ij}(\a)\langle a_\a\xi_i,\eta_j\rangle\bigg|\\
=&\bigg|\sum_{i,j;r,s}\bigg(\sum_{\a}f_{ij}(\a)a_{rs}(\a)\bigg)\xi_r(i)\overline{\eta_s(j)}\bigg|\\
=&|\langle f(a)\tilde\xi, \tilde\eta\rangle|\leq \|f(a)\|\|\tilde\xi\|\|\tilde\eta\|
\leq\|f\|_{\bm_n(\ell^1(\bz^2))}\|\tilde a\|\|\tilde\xi\|\|\tilde\eta\|\\
\leq&\|f\|_{\bm_n(\ell^1(\bz^2))}\|a\|_{M\overline{\otimes}\ell^\infty(\bz^2)}\bigg(\sum_{i=1}^n\|\xi_i\|^2\bigg)^{1/2}\bigg(\sum_{i=1}^n\|\eta_i\|^2\bigg)^{1/2},
\end{align*}
and the assertion follows.

Fix now $f\in\bm_n(\ell^1(\bz^2))$, which can be supposed to have finite support, such that $\|f\|_{\bm_n(\ell^1(\bz^2))}\leq1$.
After noticing that $\{u_\a\}_{\a\in\bz^2}$ gives rise to an element $u\in M\overline{\otimes}\ell^\infty(\bz^2)$ with unit norm, we consider a generic element $L_x\in\bm_m(M_*)$ for $x\in\bm_m(M)$ and compute
\begin{align*}
f\big(\cf_{1,\infty}(L_x)\big)_{(i,j)(k,l)}=&\sum_{\a}f_{jl}(\a)\om(u_\a^*x_{ik})
=\om\left(\bigg(\sum_{\a}f_{jl}(\a)u_\a^*\bigg)x_{ik}\right)\\
=&\om\big(F(u^*)_{jl}x_{ik}\big)=L_{x_{ik}}\big(F(u^*)_{jl}\big).
\end{align*}
Thus, by reasoning as in the proof of Proposition \ref{ohcb}, we have
$$
\left\|f\big(\cf_{1,\infty}(L_x)\big)\right\|_{\bm_{mn}}\leq\left\|L_x\right\|_{\bm_{m}(L^1)}\|F\|=\left\|L_x\right\|_{\bm_{m}(L^1)}.
$$
By taking the supremum on the r.h.s. for $f$ in the unit ball of $\bm_n(\ell^1(\bz^2))$ and on $n$, we obtain the assertion by noticing that elements of the type $L_x$ as above provide a set dense in norm of $\bm_n(M_*)$.

For the case $p=2$, pick $\xi\in\bm_n(\ch_\om)$ and consider the matrices 
$\langle\xi,\xi\rangle,\,\langle\widehat{\xi},\widehat{\xi}\rangle\in\bm_{n^2}$ whose entries are 
given by
$$
\langle\xi,\xi\rangle_{(i,k),(j,l)}:=\langle\xi_{ij},\xi_{kl}\rangle,\quad\langle\widehat{\xi},\widehat{\xi}\rangle_{(i,k),(j,l)}:=\langle\widehat{\xi_{ij}},\widehat{\xi_{kl}}\rangle.
$$
After denoting by $P_{v}$ the orthogonal projection onto the one dimensional subspace spanned by the vector $v\in\ch_\om$, by \eqref{for201}
we compute
\begin{align*}
\langle\widehat{\xi_{ij}},\widehat{\xi_{kl}}\rangle
=&\sum_{m,n\in\bz}\langle \xi_{ij}, e^{mn}\rangle\langle e^{mn}, \xi_{kl}\rangle
=\sum_{m,n\in\bz}\langle P_{e^{mn}}\xi_{ij},\xi_{kl}\rangle\\
=&\bigg\langle\bigg(\sum_{m,n\in\bz}P_{e^{mn}}\bigg)\xi_{ij},\xi_{kl}\bigg\rangle=\langle\xi_{ij},\xi_{kl}\rangle,\,\,\,i,j,k,l\in\bz,
\end{align*}
because the $e^{mn}$ provide a basis for $\ch_\om$, and thus $\big(\sum_{m,n\in\bz}P_{e^{mn}}\big)\uparrow I_{\ch_\om}$ in the strong operator topology. Consequently,
$\|\langle\widehat{\xi},\widehat{\xi}\rangle\|_{\bm_{n^2}}=\|\langle\xi,\xi\rangle\|_{\bm_{n^2}}$ and therefore $\|\cf\|^{({\rm cb})}_{2,2}= 1$ by \cite{F}, Proposition 1.
\end{proof}
\begin{rem}
When $p=2$, we have proved that $\cf_{2,2}$ is indeed a complete isometry between the canonical operator space structures of
$L_0^2(M)$ and $\ell^2(\bz^2)$, both coinciding with the Pisier OH-structure ({\it cf.} \cite{Pi}), see \cite{F}, Theorem 5.
\end{rem} 
Summarising, we have the noncommutative version of the Hausdorff-Young Theorem in our setting.
\begin{thm}
\label{hy}
Let $p\in[1,2]$ with $q\in[2,+\infty]$ its conjugate exponent. The Fourier transform restricts to a complete contraction 
$$
\cf_{p,q}: L_0^p(M)\rightarrow\ell^q(\bz^2),
$$
which is indeed a complete isometry for $p=2$.
\end{thm}
\begin{proof}
By Proposition \ref{cclp}, the assertion follows as $C_\th=C^\th$ (equal norms, Theorem 1.5 in \cite{K}), and Theorem 2 in \cite{F}) provide exact interpolation functors, see {\it e.g.}  \cite{BL}, Theorems 4.1.2 and 4.1.4.
\end{proof}
Notice that the last theorem can be viewed as a noncommutative version of the Riesz-Thorin Theorem.

The computation explained above allows to consider the noncommutative version of the Hausdorff-Young Theorem for the Fourier anti-transform defined as follows. In fact, fix any element $f\in\ell^1(\bz^2)$ and define $\widecheck{f}\in M$ given by
\begin{equation}
\label{aft1}
\widecheck{f}:=\sum_{k,l\in\bz} f(k,l)u_{kl}.
\end{equation}
The definition for the Fourier anti-transform for $p=2$ proceeds as follows. For each $f\in\ell^1(\bz^2)$ we consider the maps, the latter being the left embedding of $M$ in $\ch_\om$:
$$
f\in\ell^1(\bz^2)\mapsto\widecheck{f}\in M\mapsto \widecheck{f}\xi_\om=\sum_{k,l\in\bz}f(k,l)e^{kl}\in\ch_\om.
$$
Therefore, 
$\ell^2(\bz^2)\ni f\mapsto \widecheck{f}\in\ch_\om$ assumes the form
$$
\widecheck{f}:=\sum_{k,l\in\bz} f(k,l)e^{kl}.
$$
Concerning the $L^p$-spaces, $p\in[2,+\infty]$, the Fourier anti-transform is then associated to the left embedding of $M$ inside $M_*$. 

Denoting by $T$ such a Fourier anti-transform, we have
\begin{prop}
\label{cclp1}
The Fourier anti-transforms $T_{1,\infty}$ and $T_{2,2}$ are well defined as complete contractions w.r.t. the canonical operator space structures for the noncommutative (left) $L^p$-spaces obtained by interpolation.
\end{prop}
\begin{proof}
For the case $p=1$, we proceed as in Proposition \ref{cclp}. Firstly we note that, if $f\in\ell^1(\bz^2)$ then
$$
\|\widecheck{f}\|_{L^\infty}=\bigg\|\sum_{\a\in\bz^2}f(\a)u_a\bigg\|\leq\sum_{\a\in\bz^2}|f(\a)|=\|f\|_{\ell^1},
$$
and therefore the Fourier anti-transform is well defined as a bounded map between $\ell^1(\bz^2)$ and $L^\infty(M)$.

To check complete boundedness, fix $f\in\bm_n(\ell^1(\bz^2))$ and $g\in\bm_m(M_*)$ with $\|g\|_{\bm_m(M_*)}\leq1$,
With $F$ as before, after recalling that $\|F\|\leq\|f\|$, we compute
$$
g(\widecheck{f})_{(i,j)(k,l)}=g_{jl}\bigg(\sum_{\a}f_{ik}(\a)u_\a\bigg)
=g_{jl}\big(F(u)_{ik}\big).
$$
Therefore, $\|g(\widecheck{f})\|_{\bm_{nm}}\leq\|f\|_{\bm_n(\ell^1(\bz^2))}$
that is $T_{1,\infty}$ is completely bounded.

For $p=2$, with $f\in\ell^2(\bz^2)$ we compute
\begin{align*}
\langle \widecheck{f_{ik}},&\widecheck{f_{jl}}\rangle=\left\langle\sum_{m,n\in\bz} f_{ik}(m,n)e^{mn},\sum_{r,s\in\bz}f_{jl}(r,s)e^{rs}\right\rangle\\
=&\sum_{m,n;r,s\in\bz} f_{ik}(m,n)\overline{f_{jl}(r,s)}\langle e^{mn},e^{rs}\rangle\\
=&\sum_{m,n;r,s\in\bz} f_{ik}(m,n)\overline{f_{jl}(r,s)}\d_{m,r}\d_{n,s}=\langle f_{ik}, f_{jl}\rangle.
\end{align*}
Therefore, $T_{2,2}$ is a complete isometry concerning the left embedding.
\end{proof}
\begin{thm}
\label{hy1}
Let $p\in[1,2]$ with $q\in[2,+\infty]$ its conjugate exponent. The Fourier anti-transform extends to a
complete contraction 
$$
T_{p,q}: \ell^p(\bz^2)\rightarrow L_0^q(M),
$$
which is indeed a complete isometry for $p=2$.
\end{thm}
\begin{proof}
The proof follows analogously to that of Theorem \ref{hy}, by taking into account Proposition \ref{cclp1}.
\end{proof}

\section{the noncommutative version of the Fejer theorem}

\noindent

In our context, we provide a noncommutative version of the Fejer Theorem based on the Cesaro mean of the Fourier coefficients. By using a suitable version of the so-called "transference method", we reduce the matter to the classical convolution with the Fejer kernel on $\bt^2$ as for the classical ({\it e.g.} \cite{P}) and the tracial ({\it cf.} \cite{CXY}) cases. 
We also get the same result for the Abel mean of the Fourier coefficients by reducing the matter to the convolution with the Poisson kernel, and argue that our analysis can be used to manage all cases listed in Section 3 of \cite{CXY} and many others. 
All such cases will provide inversion formulas for the Fourier transform in $L^p$-spaces, $p=[1,2]$, arising from non type $\ty{II_1}$ representations (and also for type $\ty{II_1}$ ones when the modular operator is not trivial including the cases treated in \cite{CM}) of the noncommutative 2-torus. 

We start by noticing that the usual transference method, used in \cite{CXY} for the tracial case, cannot be directly applied to the situations under consideration. This is because the 2-parameter group of $*$-automorphisms of $\ba_{2\a}$ determined by 
$$
\ba_{2\a}\ni(U,V)\mapsto (zU,wV)\in\ba_{2\a},\quad z,w\in\bt,
$$
does not leave invariant the reference state $\om$. Yet, we can get the expected result, even for the situation studied in the present paper.

Consider the action of $\bz$ on $L^\infty(\bt,m)$ generated by the powers of the $*$-automorphism $\b$ given by
$$
\b(H)(z):=H(\mathpzc{f}^{-1}(z)),\quad H\in L^\infty(\bt,m),\,\,\, z\in\bt.
$$
The $*$-isomorphism
$$
\pi_\om(\ba_{2\a})''\sim L^\infty(\bt,m)\ltimes_{\b}\bz
$$
is realised as follows ({\it cf.} \cite{FS}, Section 6). Let $\pi_o:L^\infty(\bt,m)\to\pi_\om(\ba_{2\a})''$ be the injective $*$-homomorphism given by
$$
\pi_o(H)_n:=M_{\b^{-n}(H)}=M_{H\circ\mathpzc{f}^{n}},\quad H\in L^\infty(\bt,m),\,\,n\in\bz,
$$
where $M_h$ is the multiplication operator for the function $h$. In addition, for the shift $(\l H)_n:=H_{n-1}$, $n\in\bz$, we have
$$
\l^k\pi_o(H)\l^{-k}=\pi_o\big(\b^k(H)\big),\quad k\in\bz.
$$
Therefore, $\pi_\om(\ba_{2\a})''$ is generated by the $\pi_o(H)$ and the powers of $\l$:
$$
\pi_\om(\ba_{2\a})''=\bigg\{\sum_{|k|\leq l}\pi_o(H_k)\l^k\mid H_k\in L^\infty(\bt,m),\,\,l\in\bn\bigg\}'',
$$
see {\it e.g.} \cite{T}, Section X.1, or \cite{BR},  Section 2.7.1.

Now, let $x=\sum_{k}\pi_o(H_k)\l^k$ be a linear generator of $M=\pi_\om(\ba_{2\a})''$ with
$$
x_n=\sum_kM_{H_k\circ\mathpzc{f}^n}\l^k,\quad n\in\bz.
$$
For ${\bf w}=(w_1,w_2)\in\bt^2$, we consider the following maps $\r_{\bf w}$ defined on such generators of $M$ by
\begin{equation}
\label{trans}
\r_{\bf w}\bigg(\sum_kM_{H_k\circ\mathpzc{f}^n}\l^k\bigg):=\bigg(\sum_kw_2^{k}M_{H_k\circ\mathpzc{f}^n\circ R_{w_1}}\l^k\bigg),
\end{equation}
where for $w\in\bt$, $R_w$ denotes, with an abuse of notation, the map $R_w(z):=wz$, $z\in\bt$.
\begin{prop}
\label{pinccz}
The maps \eqref{trans} extend to $*$-automorphisms of $\pi_\om(\ba_{2\a})''$. In addition, they uniquely define contractions
$$
\r^{(p,0)}_{\bf w}:L_0^p(M)\to L_0^p(M),\quad 1\leq p<+\infty,
$$
satisfying 
\begin{equation}
\label{traions}
i_{q,p}\r^{(q,0)}_{\bf w}=\r^{(p,0)}_{\bf w} i_{q,p},\quad 1\leq p\leq q\leq+\infty.
\end{equation}
\end{prop}
\begin{proof}
Notice that $\r_{\bf w}=\r_{1,w_1}\r_{2,w_2}=\r_{2,w_2}\r_{1,w_1}$, where
\begin{align*}
\r_{1,w}&\bigg(\sum_kM_{H_k\circ\mathpzc{f}^n}\l^k\bigg):=\bigg(\sum_kM_{H_k\circ\mathpzc{f}^n\circ R_{w}}\l^k\bigg),\\
\r_{2,w}&\bigg(\sum_kM_{H_k\circ\mathpzc{f}^n}\l^k\bigg):=\bigg(\sum_kw^kM_{H_k\circ\mathpzc{f}^n}\l^k\bigg).
\end{align*}
It is easily seen that $\r_{2,w}$ define automorphisms of $M$. By Theorem X.1.7 in \cite{T}, $\r_{1,w}$ also define automorphisms of $M$. 

The pre-transpose map 
$$
{}_t\r_{\bf w}:=\r_{\bf w}^t\lceil_{M_*},
$$ 
also defines an isometry of $M_*=L_0^1(M)$ satisfying
$$
{}_t\r_{\bf w}\lceil_{M}=\r_{\bf w},
$$ 
because $M\sim i_{\infty,1}(M)\equiv L_M=L_0^\infty(M)\subset L_0^1(M)$.
Therefore, by interpolation ({\it cf.} \cite{BL}), $\r_{\bf w}$ define contractions 
$$
\r^{(p,0)}_{\bf w}:L^p_0(M)\to L^p_0(M)
$$ 
satisfying \eqref{traions}. 
\end{proof}
We note that it is possible to verify by direct inspection that the $\r^{(2,0)}_{\bf w}$ are isometric isomorphisms. Indeed, if $x\in\bigoplus_\bz L^2(\bt,m)$, for ${\bf w}=(w_1,w_2)\in\bt^2$ we get
$$
\big\|\r_{\bf w}^{(2,0)}(x)\big\|^2=\sum_{n\in\bz}\int_\bt\big|w_2^nx_n(w_1z)\big|^2\di m(z)=\sum_{n\in\bz}\int_\bt\big|x_n(z)\big|^2\di m(z)=\|x\|^2.
$$
The following results are crucial in the sequel.
\begin{lem}
\label{wts}
Let $x\in L_0^p(M)$, and ${\bf w}=(w_1,w_2)\in\bt^2$, then 
$$
\widehat{\r^{(p,0)}_{\bf w}(x)}(k,l)=w_1^{-l}w_2^{-k}\widehat{x}(k,l),\quad k,l\in\bz.
$$
\end{lem}
\begin{proof}
If $p>1$, one has (and it can be taken as the definition of the Fourier transform for $p>2$)
$$
\widehat{x}(k,l)=\widehat{i_{p,1}(x)}(k,l)=i_{p,1}(x)(u_{kl}^*),
$$
where if $x\in M=L_0^\infty(M)$, $\widehat{x}=\widehat{L_x}\equiv \widehat{i_{\infty,1}(x)}$.
Thus, we can restrict the analysis to $p=1$. 

For such a purpose, for each $x\in L_0^1(M)$ and $\eps>0$, we take $x_\eps\in M$
such that $\|x-i_{\infty,1}(x_\eps)\|<\eps$.
As $\eps$ is arbitrary and
\begin{align*}
\big|w_1^{-l}w_2^{-k}\widehat{x}(k,l)-&\widehat{\r^{(1,0)}_{\bf w}(x)}(k,l)\big|
\leq\big|\widehat{x}(k,l)-\widehat{i_{\infty,1}(x_\eps)}(k,l)\big|\\
+&\big|w_1^{-l}w_2^{-k}\widehat{i_{\infty,1}(x_\eps)}(k,l)-\widehat{\r^{(1,0)}_{\bf w}(i_{\infty,1}(x_\eps))}(k,l)\big|\\
+&\big|\widehat{\r^{(1,0)}_{\bf w}(i_{\infty,1}(x_\eps))}(k,l)-\widehat{\r^{(1,0)}_{\bf w}(x)}(k,l)\big|\\
\leq&\|x-i_{\infty,1}(x_\eps)\|_1\\
+&\big|w_1^{-l}w_2^{-k}\widehat{i_{\infty,1}(x_\eps)}(k,l)-\widehat{\r^{(1,0)}_{\bf w}(i_{\infty,1}(x_\eps))}(k,l)\big|\\
+&\big\|\r^{(1,0)}_{\bf w}\big(i_{\infty,1}(x_\eps)-x\big)\big\|_1\\
\leq&2\eps+\big|w_1^{-l}w_2^{-k}\widehat{i_{\infty,1}(x_\eps)}(k,l)-\widehat{\r^{(1,0)}_{\bf w}(i_{\infty,1}(x_\eps))}(k,l)\big|,
\end{align*}
the assertion will follow if we prove it for the generators of $L^1(M)=M_*$ of the form $L_{a}\equiv i_{\infty,1}(a)$.
We first note that
\begin{equation}
\label{trcfi}
\r_{\bf w}(u_{kl})\xi_\om=w_1^{l}w_2^{k}u_{kl}\xi_\om,
\end{equation}
and then
\begin{align*}
\widehat{\r^{(1,0)}_{\bf w}\big(L_{a}\big)}(k,l)=&L_{a}\big(\r_{\bf w}(u_{kl})^*\big)
=\langle a\xi_\om,\r_{\bf w}(u_{kl})\xi_\om\rangle\\
=&w_1^{-l}w_2^{-k}\langle a\xi_\om,u_{kl}\xi_\om\rangle
=w_1^{-l}w_2^{-k}\widehat{L_{a}}(k,l).
\end{align*}
\end{proof}
\begin{prop}
\label{pisy}
For each $1\leq p\leq2$, we get
$$
L^p\!\!-\!\!\lim_{{\bf w}\to {\bf 1}}\r^{(p,0)}_{\bf w}x=x,\quad x\in L_0^p(M).
$$
\end{prop}
\begin{proof}
For $x\in L_0^2(M)\sim\ch_\om=\oplus_{n\in\bz} L^2(\bt,\di m)$, we get by the Lebesgue dominated convergence theorem,
$$
\big\|\r^{(2,0)}_{\bf w}x-x\big\|_2^2=\sum_{n\in\bz}\int_\bt\di m(z)\big|w^{n}_2x_n(w^{1}_1z)-x_n(z)\big|^2\to0,
$$
as ${\bf w}\to {\bf 1}\equiv(1,1)\in\bt^2$. Thus, the result holds true for $p=2$. For $1\leq p<2$ and $\eps>0$, let now $x\in L_0^p(M)$. %Denote $i_{2,p}$ the embedding of $L_0^2(M)$ into $L_0^p(M)$.
As $i_{2,p}(L^2)$ is dense in each such $L^p$-spaces, we pick $x_\eps\in L_0^2(M)$ such that $\|x-i_{2,p}(x_\eps)\|_p<\eps$. Again by Proposition \ref{pinccz}, we also have 
$\big\|\r^{(p,0)}_{\bf w}(x-i_{2,p}(x_\eps))\big\|_p<\eps$ for each 
${\bf w}\in\bt^2$. We then compute by the Holder inequality,
\begin{align*}
\big\|\r^{(p,0)}_{\bf w}x-x\big\|_p<&2\eps+\big\|\r^{(p,0)}_{\bf w}(i_{2,p}(x_\eps))-i_{2,p}(x_\eps)\big\|_p\\
\leq&2\eps+\big\|i_{2,p}\big(\r^{(2,0)}_{\bf w}(x_\eps)-x_\eps\big)\big\|_p\\
\leq&2\eps+\big\|\r^{(2,0)}_{\bf w}x_\eps-x_\eps\big\|_2.
\end{align*}
As $\eps>0$ is arbitrary, by the previous part $\big\|\r^{(p,0)}_{\bf w}x-x\big\|_p\to0$ whenever ${\bf w}\to {\bf 1}$.
\end{proof}
Here, there is the noncommutative version of the Fejer theorem in our context. It will provide an inversion formula for the Fourier transform for $1\leq p\leq 2$.
\begin{thm}
\label{fece}
Let $1\leq p\leq 2$ and $x\in L_0^p(M)$. Then we get,
$$
L^p\!\!-\!\!\lim_{N\to+\infty}\sum_{|k|,|l|\leq N}\bigg(1-\frac{|k|}{N+1}\bigg)\bigg(1-\frac{|l|}{N+1}\bigg)\widehat{x}(k,l)i_{\infty,p}(u_{kl})=x.
$$
\end{thm}
\begin{proof}
We first prove that, for each fixed $x\in L^p_0(M)$ and $N=1,2,\dots$, 
\begin{align*}
&\sum_{|k|,|l|\leq N}\bigg(1-\frac{|k|}{N+1}\bigg)\bigg(1-\frac{|l|}{N+1}\bigg)\widehat{x}(k,l)i_{\infty,p}(u_{kl})\\
=&\int_{\bt^2}\di m(w_1)\di m(w_2)\F_N(w_1)\F_N(w_2)\r^{(p,0)}_{w_1,w_2}(x),
\end{align*}
where $\F_N$ is the Fejer kernel. By a standard 3$\eps$-argument, it is enough to check the identity for the total set (in $L^p_0(M)'\sim L_1^q(M)$, $q$ being the conjugate exponent) of elements of the type  
$j_{\infty,q}(u^*_{rs})$, $r,s\in\bz$. 
By Lemma \ref{wts},
$$
\r^{(p,0)}_{w_1,w_2}(x)\big(j_{\infty,q}(u^*_{rs})\big)=\widehat{\r^{(p,0)}_{w_1,w_2}(x)}(r,s)=w_1^{-s}w_2^{-r}\widehat{x}(r,s),
$$
and thus
\begin{align*}
&\int_{\bt^2}\di m(w_1)\di m(w_2)\F_N(w_1)\F_N(w_2)\r^{(p,0)}_{w_1,w_2}(x)\big(j_{\infty,q}(u^*_{rs})\big)\\
=&\widehat{x}(r,s)\int_{\bt}\di m(w_1)\F_N(w_1)w_1^{-s}\int_{\bt}\di m(w_2)\F_N(w_2)w_1^{-r}\\
=&\widehat{\F_N}(r)\widehat{\F_N}(s)\widehat{x}(r,s)\\
=&\bigg(1-\frac{|r|}{N+1}\bigg)\bigg(1-\frac{|s|}{N+1}\bigg)\widehat{x}(r,s)\\
=&\sum_{|k|,|l|\leq N}\bigg(1-\frac{|k|}{N+1}\bigg)\bigg(1-\frac{|l|}{N+1}\bigg)\widehat{x}(k,l)\om(u^*_{rs}u_{kl})\\
=&\sum_{|k|,|l|\leq N}\bigg(1-\frac{|k|}{N+1}\bigg)\bigg(1-\frac{|l|}{N+1}\bigg)\widehat{x}(k,l)i_{\infty,p}(u_{kl})\big(j_{\infty,q}(u^*_{rs})\big).
\end{align*}
The proof now follows by Proposition \ref{pisy} as the Fejer kernel $\F_N(z)$ is an approximate identity ({\it e.g.} \cite{P}) on $\bt$.
\end{proof}
We report the following generalisation corresponding to the Abel mean associated to the Poisson kernel, whose proof is completely analogous to the previous one. The details relative to the remaining cases in \cite{CXY} are left to the reader.
\begin{thm}
\label{abt1}
Let $1\leq p\leq 2$ and $x\in L_0^p(M)$. Then we get,
$$
L^p\!\!-\!\!\lim_{r\uparrow1}\sum_{k,l\in\bz}r^{|k|+|l|}\widehat{x}(k,l)i_{\infty,p}(u_{kl})=x.
$$
\end{thm}
\begin{proof}
We note that, as before,
\begin{align*}
&\sum_{k,l\in\bz}r^{|k|+|l|}\widehat{x}(k,l)i_{\infty,p}(u_{kl})\\
=&\int_{\bt^2}\di m(w_1)\di m(w_2)P_r(w_1)P_r(w_2)\r^{(p,0)}_{w_1,w_2}(x),
\end{align*}
where $P_r$, $0\leq r<1$ is the Poisson kernel, which is again an approximate identity on $\bt$. The assertion follows again by Proposition \ref{pisy}.
\end{proof}

\section{an alternative definition of the fourier transform}
\label{altfou}

\noindent

Even if it is still unclear how to construct a one-parameter group of Fourier transforms, one for each embedding in \eqref{em10} of $M$ in its predual $M_*$, we are able to provide an alternative definition of the Fourier transform, which is essentially associated to the right embedding 
$$
M\ni a\to j_{\infty,1}(a)=R_a\in M_*,
$$ 
in the following way. 
\begin{defin}
For $x\in M_*=L_1^1(M)$ of the form $R_a$, we define the sequence $\wideparen{x}(k,l)$ of the Fourier coefficients as
\begin{equation}
\label{forrztf}
\wideparen{x}(k,l):=\om(au_{kl}),\quad k,l\in\bz.
\end{equation}
\end{defin}
It is immediate to check that, for $a\in M$ we get
\begin{align*}
\wideparen{L_a}(k,l)&=\om(u_{kl}a)\neq\om(u^*_{kl}a)=\widehat{L_a}(k,l),\\
\widehat{R_a}(k,l)&=\om(au^*_{kl})\neq\om(au_{kl})=\wideparen{R_a}(k,l).
\end{align*}
This simply means that, for general $f\in M_*=L^1_0(M)=L^1_1(M)$ and $k,l\in\bz$, we get
$$
\widehat{f}(k,l)=f(u^*_{kl})\neq f(u_{kl})=\wideparen{f}(k,l).
$$
For $p=2$, suppose for simplicity that $\xi\in\cd_{\D_\om^{1/2}}$. Then 
$$
L^1_1(M)\ni x_\xi:=j_{2,1}(\xi)=\langle\,{\bf\cdot}\,\D_\om^{1/2}\xi,\xi_\om\rangle.
$$
Therefore, for $k,l\in\bz$ we obtain by $J_\om \D^{1/2}_\om=\D^{-1/2}_\om J_\om$ ({\it cf.} \eqref{tomis}),
\begin{equation}
\label{1le2}
\wideparen{x_\xi}(k,l)=\langle u_{kl}\D_\om^{1/2}\xi,\xi_\om\rangle=\langle\xi,J_\om e^{kl}\rangle,
\end{equation}
which is meaningful for all $\xi\in\ch_\om$.

By putting 
$$
\eeps^{kl}:=J_\om e^{kl},\quad k,l\in\bz,
$$
we note that the $\eeps^{kl}$ still provide an orthonormal basis for $\ch_\om$ as the conjugation $J_\om$ is an anti-unitary involution. For the Fourier transform \eqref{1le2} of a generic vector
$\xi\in\ch_\om$, we put
$$
\wideparen{\xi}(k,l):=\langle\xi,\eeps^{kl}\rangle,\quad  k,l\in\bz.
$$
The corresponding Fourier anti-transform 
$$
f\in\ell^1(\bz^2)\to\breve{f}\in M=L^\infty_1(M),
$$
that is the analogous of \eqref{aft1}, is defined as
\begin{equation}
\label{1fao2}
\breve{f}:=\sum_{k,l\in\bz}f(k,l)u^*_{kl}.
\end{equation}
By taking into account
$$
L^\infty_1(M)=M\ni a\hookrightarrow\D_\om^{1/2}a\xi_\om\in\ch_\om=L^2_1(M)
$$
and $J_\om \D^{1/2}_\om=\D^{-1/2}_\om J_\om$ as before, for the Fourier anti-transform 
$$
f\in\ell^2(\bz^2)\to\breve{f}\in L^2_1(M),
$$
we obtain
$$
\breve{f}=\sum_{k,l\in\bz}f(k,l)\D_\om^{1/2}u^*_{kl}\xi_\om=\sum_{k,l\in\bz}f(k,l)\eeps^{kl},\,\,k,l\in\bz.
$$
Therefore, the Fourier anti-transform $\!\breve{}$ is also a (complete) unitary equivalence at level of Hilbert spaces as expected. 

From the last consideration, compared with the analogous one \eqref{for201}, we can recognise why $\,\widehat{{}}\,$ and $\,\wideparen{}\,$ are associated to the left and right embedding of $M$ in $M_*$, respectively. 

As noticed before, it is still unclear how we can interpolate between both definitions to obtain a one-parameter of Fourier transforms associated to the other embeddings $\iota^\th_{\infty,1}:M\hookrightarrow M_*$, $\th\in(0,1)$.

Also for this case of the Fourier transform $\,\wideparen{}\,$ and the corresponding anti-transform $\,\breve{}\,\,$, the Riemann-Lebesgue Lemma ({\it cf.} Theorem \ref{rlcz}), the Hausdorff-Young inequality ({\it cf.} Theorem \ref{hy}), and finally Theorem \ref{hy1} hold true. We leave the details of the proofs to the reader. 

The corresponding results analogous to Theorem \ref{fece} and Theorem \ref{abt1} hold true
as well. For this situation, we report the crucial details of the proof.
\begin{thm}
Let $1\leq p\leq 2$ and $x\in L_0^p(M)$. Then 
\begin{align*}
L^p\!\!-\!\!\lim_{N\to+\infty}\sum_{|k|,|l|\leq N}&\bigg(1-\frac{|k|}{N+1}\bigg)\bigg(1-\frac{|l|}{N+1}\bigg)\wideparen{x}(k,l)j_{\infty,p}(u^*_{kl})=x,\\
L^p\!\!-\!\!\lim_{r\uparrow1}\sum_{k,l\in\bz}&r^{|k|+|l|}\wideparen{x}(k,l)j_{\infty,p}(u^*_{kl})=x.
\end{align*}
\end{thm}
\begin{proof}
With $\r^{(p,1)}_{\bf w}:L^p_1(M)\to L^p_1(M)$ the extension of the $*$-automorphism in \eqref{trans} to the right $L^p$-spaces, by \eqref{trcfi} we check for $x=j_{\infty,p}(a)$,
\begin{align*}
\r^{(p,1)}_{w_1,w_2}(x)\big(i_{\infty,q}(u_{rs})\big)=&\wideparen{\r^{(p,1)}_{w_1,w_2}(x)}(r,s)
=\langle a\r_{\bf w}(u_{rs})\xi_\om,\xi_\om\rangle\\
=&w_1^{s}w_2^{r}\langle au_{rs}\xi_\om,\xi_\om\rangle
=w_1^{s}w_2^{r}\wideparen{x}(r,s).
\end{align*}
Therefore, by a standard 3$\eps$-argument, we argue that for a generic $x\in L^p_1(M)$,
\begin{align*}
&\sum_{|k|,|l|\leq N}\bigg(1-\frac{|k|}{N+1}\bigg)\bigg(1-\frac{|l|}{N+1}\bigg)\wideparen{x}(k,l)j_{\infty,p}(u^*_{kl})\\
=&\int_{\bt^2}\di m(w_1)\di m(w_2)\F_N(w_1)\F_N(w_2)\r^{(p,1)}_{w^{-1}_1,w^{-1}_2}(x),\\
&\sum_{k,l\in\bz}r^{|k|+|l|}\wideparen{x}(k,l)j_{\infty,p}(u^*_{kl})\\
=&\int_{\bt^2}\di m(w_1)\di m(w_2)P_r(w_1)P_r(w_2)\r^{(p,1)}_{w^{-1}_1,w^{-1}_2}(x).
\end{align*}
Then the assertions follow analogously to those of the corresponding results in the previous section.
\end{proof}
We compare the definitions of the Fourier transforms in the present paper with the original ones concerning the situation associated to the canonical trace, and for the commutative situation $C(\bt^2)$ corresponding to the case when the deformation angle $\a$ is zero.

Suppose that $A\in\ba_{2\a}$. By taking into account our definitions of the Fourier transform and that in \eqref{fzocf}, for $x:=L_{\pi_\t(A)}=R_{\pi_\t(A)}\in\pi_\t(\ba_{2a})''_*$ we get
\begin{align*}
\widehat{x}(k,l)=&e^{2\pi\imath\a kl}\t\big(W(-l,-k)A\big),\\
\wideparen{x}(k,l)=&e^{-2\pi\imath\a kl}\t\big(W(l,k)A\big),
\end{align*}
where the Weyl operators $W(m,n)$ are defined in \eqref{wcea}.

The classical case corresponds to $\a=0$, that is $\ba_0\sim C(\bt^2)$, and the underlying measure corresponds to the Haar one $m\times m=\frac{\di\th_1\di\th_2}{4\pi^2}$. In this situation, each element in $L^1(M)$ uniquely corresponds to a (equivalence class of) function $f$ in $L^1(\bt^2,m\times m)$:
$$
L^1(M)\ni x\equiv f_x\in L^1(\bt^2,m\times m)\sim L^1(M).
$$ 
Therefore, we get
\begin{align*}
\widehat{x}(k,l)=&\int_{\bt^2}f_x\big(e^{\imath\th_1},e^{\imath\th_2}\big)e^{-\imath(l\th_1+k\th_2)}\frac{\di\th_1\di\th_2}{4\pi^2}=\widehat{f_x}(l,k),\\
\wideparen{x}(k,l)=&\int_{\bt^2}f_x\big(e^{\imath\th_1},e^{\imath\th_2}\big)e^{\imath(l\th_1+k\th_2)}\frac{\di\th_1\di\th_2}{4\pi^2}=\widecheck{f_x}(l,k),
\end{align*}
where $\widehat{f}$ and $\widecheck{f}$ denote the classical Fourier transform and anti-transform of $f\in L^1(\bt^2,m\times m)$.

We end the present section by noticing the following fact arising by algebraic similarities. Even if there is no reasonable motivation to consider $\ba_{2\a}$ as a group-like quantum object with 
$\bz^2$ as its "dual object", we can consider the map \eqref{1fao2} as a Fourier transform from $\ell^1(\bz^2)$ to $M$, whose corresponding anti-transform $M_*\to\ell^\infty(\bz^2)$ is given by \eqref{forrztf}. This alternative way to see the emerging situation does not cause any trouble.

\section{deformed dirac operators and modular spectral triples}

\noindent

The present section is devoted to defining a one-parameter family of modular spectral triples for each $\eta\in[0,1]$, extending the ones defined in Section 9 of \cite{FS}. As an application of the previous analysis, we show how the Fourier transforms defined here, allow us to "diagonalise" the corresponding Dirac operators $\mathrm{D}^{(\eta)}$, $\eta=\{0,1/2,1\}$.

For a fixed diffeomorphism $\mathpzc{f}$ of the unit circle $\bt$ with growth sequence $\G_n(\mathpzc{f})$ as in Section \ref{2dixf}, we start with an undeformed Dirac operator
by putting
$$
\mathrm{D}_n=\begin{pmatrix} 
	 0 &\mathrm{L}_n\\
	\mathrm{L}_n^*& 0\\
     \end{pmatrix}:=\begin{pmatrix} 
	 0 &\left(\imath z\frac{\di\,\,}{\di z}-a_nI\right)\\
	\left(-\imath z\frac{\di\,\,}{\di z}-a_nI\right)& 0\\
     \end{pmatrix},
$$
with $a_0=0$ and
$$
a_n:=\sign(n)\sum_{l=1}^{|n|}\frac1{\G_{l-\frac{1-\sign(n)}2}(\mathpzc{f})},\quad n\in\bz\backslash\{0\}.
$$
Here, 
$$
\sign(n):=\left\{\!\!\!\begin{array}{ll}
                      -1\!\!&\text{if}\,\, n<0,\\
                     \,\,\,\,\,1\!\!&\text{if}\,\, n\geq0,
                    \end{array}
                    \right.
$$
and $\mathrm{L}=\bigoplus_{n\in\bz}\left(\imath z\frac{\di\,\,}{\di z}-a_nI\right)$ is precisely the operator appearing in Definition \ref{fmst} of the modular spectral triple.
To simplify notation, we drop the subscript "$\mathrm{L}$" in the foregoing analysis.

Consider the Sobolev-Hilbert space
$$
H^1(\bt):=\big\{f\in AC(\bt)\mid f'\in L^2(\bt,m)\big\},
$$ 
where $AC(\bt)$ denotes the set of all absolutely continuous complex valued functions on the unit circle.

For $\eta\in[0,1]$ and $\d_n$ as in \eqref{rcanz}, 
we can define deformed Dirac Operators as
\begin{align*}
\mathrm{D}^{(\eta)}=&\bigoplus_{n\in\bz}\begin{pmatrix} 
	 0 &M_{\d_n^{\eta-1}}\mathrm{L}_nM_{\d_n^{-\eta}}\\
	M_{\d_n^{-\eta}}\mathrm{L}_n^*M_{\d_n^{\eta-1}}& 0\\
     \end{pmatrix}\\
     =&\bigoplus_{n\in\bz}\mathrm{D}^{(\eta)}_n
    =\begin{pmatrix} 
	 0 &\D_\om^{\eta-1}\mathrm{L}\D_\om^{-\eta}\\
	\D_\om^{-\eta}\mathrm{L}^*\D_\om^{\eta-1}& 0\\
     \end{pmatrix},\\
\end{align*}
on their natural domain
$$
\cd:=\bigg\{\xi\in\bigoplus_{n\in\bz}H^1(\bt)\oplus H^1(\bt)\mid\sum_{n\in\bz}\|\mathrm{D}^{(\eta)}_n\xi_n\|^2<+\infty\bigg\}.
$$
As explained in \cite{FS}, each $\mathrm{D}^{(\eta)}_n$ is invertible with bounded inverse, after defining for $n=0$,
$$
(\mathrm{D}^{(\eta)}_0)^{-1}:=(\mathrm{D}^{(\eta)}_0)^{-1}P^\perp_{\ker\left(\mathrm{D}^{(\eta)}_0\right)}.
$$
In addition, each $(\mathrm{D}^{(\eta)}_n)^{-1}$ is compact. 

By following the lines of Theorem 9.1 of \cite{FS}, we get
\begin{prop}
If $\G_n(\mathpzc{f})=o(\ln n)$, then for each $\eta\in[0,1]$, $\mathrm{D}^{(\eta)}$ has compact resolvent.
\end{prop}
\begin{proof}
For the inverses $(\mathrm{D}^{(\eta)}_n)^{-1}$, which are all compact operators, we have for each $n\in\bz$,
$$
(\mathrm{D}^{(\eta)}_n)^{-1}=\begin{pmatrix} 
	 0 &M_{\d_n^{1-\eta}}(\mathrm{L}^*_n)^{-1}M_{\d_n^{\eta}}\\
	M_{\d_n^{\eta}}\mathrm{L}_n^{-1}M_{\d_n^{1-\eta}}& 0\\
     \end{pmatrix}.
$$
Suppose now $\G_n(\mathpzc{f})=o(\ln n)$. By following the same computations in the above mentioned theorem, we get for $|n|$ sufficiently large,
$$
\big\|(\mathrm{D}^{(\eta)}_n)^{-1}\big\|\leq\G_n(\mathpzc{f})\big\|\mathrm{D}_n^{-1}\big\|\leq\frac1{\frac1{\ln|n|}\sum_{l=2}^{|n|}\frac1{\ln l}}\to0
$$
as $n\to\infty$, and the assertion follows.
\end{proof}
We have shown in \cite{FS} that if $\mathpzc{f}$ is one of the diffeomorphisms constructed in \cite{M} for $\a$ a {\bf UL} number, then the $\mathrm{D}^{(\eta)}$ have compact resolvent. The same happens if $\a$ is Diophantine and $\mathpzc{f}$ is any diffeomorphism of the unit circle with $\r(\mathpzc{f})=2\a$.

In order to construct the associated modular spectral triples, we have to define the corresponding deformed commutators. For ${\bf t}:=(t_1,t_2)\in\br^2$, we define
$$
\S_{\bf t}(A):=
\begin{pmatrix} 
	 \s^{\om}_{t_1}(A)&0\\
	0&\s^{\om}_{t_2}(A)\\
     \end{pmatrix}.
$$
To simplify the analysis, we first suppose that $A\in M_\infty$, the set of the entire elements of $M$, under the action of the modular group. For such operators and for $\eta\in[0,1]$, we define ${\bf z}(\eta):=-\imath(1-\eta, \eta)$ and put
$$
\cd^{(\eta)}(A)\equiv\imath\big[\mathrm{D}^{(\eta)}, A\big]_{(\eta)}:=\imath\big(\mathrm{D}^{(\eta)}\S_{{\bf z}(\eta)}(A)-\S_{-{\bf z}(\eta)}(A)\mathrm{D}^{(\eta)}\big).
$$
A straightforward computation yields
\begin{equation}
\label{spetr112}
\cd^{(\eta)}(A)=\imath
\begin{pmatrix} 
	0&\D_\om^{\eta-1}[\mathrm{L},A]\D_\om^{-\eta}\\
	\D_\om^{-\eta}[\mathrm{L}^*,A]\D_\om^{\eta-1}& 0\\
     \end{pmatrix}.
\end{equation}
Notice that $\cd^{(\eta)}(A)$ can be meaningful even if $A$ is not an analytic element. In the sequel, we assume \eqref{spetr112} as the definition of the deformed derivation for each $\eta\in[0,1]$. 

Now we show that the deformed derivation in \eqref{spetr112} provides sufficiently many bounded operators.
\begin{prop}
Let $a\in\ba^o_{2\a}$. Then $\cd^{(\eta)}(\pi_\om(a))$ defines a bounded operator acting on $\ch_\om\oplus\ch_\om$.
\end{prop}
\begin{proof}
We sketch the proof, and refer the reader to Section 9 and 11 of \cite{FS} for further details. We first note that it is enough to check the assertion for the generators of 
$\ba^o_{2\a}$, and in particular for the one step shift $\l$ and its inverse $\l^{-1}$. By considering the adjoints and all $\eta\in[0,1]$, it is enough to check the boundedness for the following case, obtaining
$$
\left\|\left(\D_\om^{\eta-1}[\mathrm{L},\l]\D_\om^{-\eta}g\right)_n\right\|
\leq|a_{n-1}-a_n|\G_{|n|}(\mathpzc{f})^{1-\eta}\G_{|n-1|}(\mathpzc{f})^{\eta}\|g\|.
$$ 
By taking into account that $|a_{n-1}-a_n|=1/\G_{|n|}(\mathpzc{f})$ and $\G_{l}(\mathpzc{f})/\G_{l\pm1}(\mathpzc{f})\leq\G_{1}(\mathpzc{f})$, the assertion follows.
\end{proof}
Notice that, in \cite{FS} it is shown that $\cd^{(\eta)}(\pi_\om(a))$ uniquely defines a bounded operator for elements $a$ in a dense $*$-subalgebra of $\ba_{2\a}$, and thus we can define (nontrivial) modular spectral triples associated to the twisted Dirac operator $\mathrm{D}^{(\eta)}$. Indeed, the modular spectral triple satisfying (i)-(iv) in Definition \ref{fmst} is given by 
$\big(\om_\m, \ba_{2\a}^o,\mathrm{L}\big)_\eta$, provided $\G_n(\mathpzc{f})=o(\ln n)$, which happens if $\mathpzc{f}$ is a diffeomorphism constructed according to Proposition 3.1 of \cite{FS} for a {\bf UL}-number $\a$, or if $\a$ is Diophantine. For the latter case, the representation $\pi_{\om_\m}$ always generates the type $\ty{II_1}$ hyperfinite factor $R$, whereas for the former the type is determined by the ratio set $r\big(R_{2\a},[\m]\big)$ (or equivalently by $r\big(\mathpzc{f},[m]\big)$) if it is not of type $\ty{II_1}$.
We also note that the $\cd^{(\eta)}$ define $*$-maps on their natural (common) domain:
$$
\cd^{(\eta)}(\pi_\om(a))^*=\cd^{(\eta)}(\pi_\om(a^*)),\quad \eta\in[0,1],\,\,a\in\ba^o_{2\a}.
$$
We now show how the previously defined Fourier transform in Section \ref{sec3} diagonalises $\mathrm{D}^{(\eta)}$ for $\eta=0,1$, while the one defined in Section \ref{altfou} diagonalises 
$\mathrm{D}^{(1/2)}$. Indeed, denote $\cf:\ch_\om\to\ell^2(\bz^2)$ one of the Fourier transforms defined above. Then 
$\big(\cf\oplus\cf\big)\mathrm{D}^{(\eta)}\big(\cf^{-1}\oplus\cf^{-1}\big)$ provides hermitian matrices 
$$
\big(\cf\oplus\cf\big)\mathrm{D}^{(\eta)}\big(\cf^{-1}\oplus\cf^{-1}\big)=
\begin{pmatrix} 
	0&A^{(\eta)}\\
	\big(A^{(\eta)}\big)^*& 0\\
     \end{pmatrix},
$$
where the $A^{(\eta)}=\big(A^{(\eta)}_{(k,l)(rs)}\big)_{k,l,r,s\in\bz}$ are (infinite) numerical matrices given, for $\eta=0,1/2,1$ and $k,l,r,s\in\bz$, by
\begin{align*}
A^{(0)}_{(k,l)(rs)}=&\langle\D_\om^{-1}\mathrm{L}e^{kl},e^{rs}\rangle\\
A^{(1)}_{(k,l)(rs)}=&\langle\mathrm{L}\D_\om^{-1}e^{kl},e^{rs}\rangle\\
A^{(1/2)}_{(k,l)(rs)}=&\langle\D_\om^{-1/2}\mathrm{L}\D_\om^{-1/2}\eeps^{kl},\eeps^{rs}\rangle.
\end{align*}     
It is then enough to compute the right upper corners $A^{(\eta)}$ of $\big(\cf\oplus\cf\big)\mathrm{D}^{(\eta)}\big(\cf^{-1}\oplus\cf^{-1}\big)$, obtaining
\begin{prop}
For $k,l,r,s\in\bz$, we have
\begin{align*}
\big\langle\D_\om^{-1}\mathrm{L}e^{kl},e^{rs}\big\rangle=&(\imath l-a_k)\widehat{1/\d_k}(s-l)\d_{k,r},\\
\big\langle\mathrm{L}\D_\om^{-1}e^{kl},e^{rs}\big\rangle=&(\imath s-a_k)\widehat{1/\d_k}(s-l)\d_{k,r},\\
\big\langle\D_\om^{-1/2}\mathrm{L}\D_\om^{-1/2}\eeps^{kl},\eeps^{rs}\big\rangle=&-\big(\imath l\d_{l,s}+a_{-k}\widehat{\d_{k}}(l-s)\big)\d_{k,r}.
\end{align*}
\end{prop}
\begin{proof}
For the first equality, we get
\begin{align*}
\big\langle\D_\om^{-1}\mathrm{L}\xi_{kl},\xi_{rs}\big\rangle=&\sum_{n\in\bz}(\imath l-a_n)\d_{k,n}\d_{r,n}\oint\frac{\di z}{2\pi\imath z}1/\d_n(z)z^{-(s-l)}\\
=&(\imath l-a_k)\widehat{1/\d_k}(s-l)\d_{k,r}.
\end{align*}
The second one follows analogously.

The third equality follows after some computations. We first note that
$$
\big\langle\D_\om^{-1/2}\mathrm{L}\D_\om^{-1/2}\eeps^{kl},\eeps^{rs}\big\rangle=\langle\mathrm{L}u^*_{kl}\xi_\om,u^*_{rs}\xi_\om\rangle.
$$
Secondly, after some computations we get
$$
\big(u^*_{kl}\xi_\om\big)_n(z)=\mathpzc{f}^n(z)^{-l}\d_{n,-k},\quad n\in\bz,\,z\in\bt.
$$
We consider the elementary changes of variables
$$
w:=\mathpzc{f}^{-k}(z),\quad k\in\bz.
$$
For the first addendum of $\mathrm{L}$, we get
\begin{align*}
\bigg\langle\bigg(\bigoplus_\bz\imath z\frac{\di\,\,}{\di z}\bigg)u^*_{kl}\xi_\om,u^*_{rs}\xi_\om&\bigg\rangle
=\imath\d_{k,r}\oint z\frac{\di\,\,}{\di z}\big(\mathpzc{f}^{-k}(z)^{-l}\big)\mathpzc{f}^{-k}(z)^{s}\frac{\di z}{2\pi\imath z}\\
=&-\imath l\d_{k,r}\oint w^{-(l-s)}\frac{\di w}{2\pi\imath w}=-\imath l\d_{k,r}\d_{l,s}.
\end{align*}
For the second addendum, after the changes of variables as above and \eqref{rcanz}, we get
\begin{align*}
&\bigg\langle\bigg(\bigoplus_{n\in\bz}a_n I\bigg)u^*_{kl}\xi_\om,u^*_{rs}\xi_\om\bigg\rangle
=a_{-k}\d_{k,r}\oint\mathpzc{f}^{-k}(z)^{-(l-s)}\frac{\di z}{2\pi\imath z}\\
=&a_{-k}\d_{k,r}\oint w^{-(l-s)}\frac{w}{\mathpzc{f}^{k}(w)(D\mathpzc{f}^{-k})(\mathpzc{f}^{k}(w))}
\frac{\di w}{2\pi\imath w}\\
=&a_{-k}\d_{k,r}\oint w^{-(l-s)}\frac{wD\mathpzc{f}^{k}(w)}{\mathpzc{f}^{k}(w)}
\frac{\di w}{2\pi\imath w}
=a_{-k}\d_{k,r}\widehat{\d_k}(l-s).
\end{align*}
Collecting together, we obtain the assertion.
\end{proof}

\section*{Acknowledgements}

The author acknowledges the financial support of Italian INDAM-GNAMPA.
The present project is part of "MIUR Excellence Department Project awarded to the Department of Mathematics, University of Rome Tor Vergata, CUP E83C18000100006".

He is grateful to the referees for a very careful reading of the manuscript, and for several suggestions which contribute to improve the presentation of the paper.

\end{document}